\documentclass[reqno,12pt,letterpaper]{amsart}
\usepackage[proof]{sdmacros}

\title[Fractal uncertainty for transfer operators]%
{Fractal uncertainty for transfer operators}
\author{Semyon Dyatlov}
\email{dyatlov@math.mit.edu}
\address{Department of Mathematics, Massachusetts Institute of Technology,
Cambridge, MA 02139}
\address{Department of Mathematics, University of California, Berkeley, CA 94720}
\author{Maciej Zworski}
\email{zworski@math.berkeley.edu}
\address{Department of Mathematics, University of California, Berkeley, CA 94720}

\newcommand{\rsa}{\rightsquigarrow}

\begin{document}

\begin{abstract}
We show directly that the fractal uncertainty principle 
of Bourgain--Dyatlov~\cite{fullgap} implies that there exists 
$ \sigma > 0 $ for which the Selberg zeta function~\eqref{eq:Selb} for a convex co-compact hyperbolic surface has only finitely many zeros with $ \Re s \geq \frac12 - \sigma$. That eliminates advanced microlocal techniques of Dyatlov--Zahl~\cite{hgap} though we stress
that these techniques are still needed for resolvent bounds and for possible generalizations to the case of non-constant curvature.
\end{abstract}

\maketitle

\addtocounter{section}{1}
\addcontentsline{toc}{section}{1. Introduction}

The purpose of this note is to give a new explanation of the connection between the {\em fractal uncertainty principle}, which is a statement in harmonic analysis, and the existence of {\em zero free strips for Selberg zeta functions}, which is a statement in geometric spectral theory/dynamical systems. 
The connection is proved via the {\em transfer operator} which is a well known object in thermodynamical formalism of chaotic dynamics.

To explain the fractal uncertainty principle
we start with its \emph{flat} version, given by~\eqref{eq:fup} below. Let $ X \subset
[ -1 , 1 ] $ be a $ \delta $-regular set in the following sense:
there exists a Borel measure $ \mu $ supported on $ X $ and a constant $ C_R$ such that 
for each interval $ I $ centered
on $X$ of size $ | I | \leq 1 $, we have $ C_R^{-1}|I|^\delta\leq \mu ( I ) \leq C_R |I|^\delta $.

Bourgain--Dyatlov~\cite[Theorem 4]{fullgap} proved that 
when $ \delta < 1 $, there exist $ \beta > 0 $ and $ C_1 $ depending only on $ \delta,C_R $ such that for all~$ f \in L^2 ( \mathbb R ) $
\begin{equation}
\label{eq:fup}
\supp \mathcal F_h f \subset X(h) \ \Longrightarrow \ 
\| f \|_{L^2 ( X(h) ) } \leq C_1 h^{\beta } \| f \|_{ L^2 ( \mathbb R ) } , 
\end{equation}
where $ \mathcal F_h f := \hat f ( \xi/h ) $ is the semiclassical Fourier transform \eqref{eq:semif}
and
$$
X(h):=X+[-h,h]
$$
denotes the $h$-neighborhood of~$X$. Roughly
speaking, \eqref{eq:fup} quantifies the statement that 
a function and its Fourier transform cannot both concentrate on a fractal set.

To explain the spectral geometry side, let $M=\Gamma\backslash\mathbb H^2$ be a convex co-compact hyperbolic surface, that is, a non-compact hyperbolic surface with a finite number of funnel ends
and no cusps -- see \cite[\S 2.4]{BorthwickBook}. 
The Selberg zeta function is defined by 
\begin{equation}
\label{eq:Selb}
Z_M(s)=\prod_{\ell\in\mathscr L_M}\prod_{k=0}^\infty \big(1-e^{-(s+k)\ell}\big),\quad\Re s\gg 1,
\end{equation}
and it continues as an entire function to $ \mathbb C $. 
Here $\mathscr L_M$ denotes the set of lengths of primitive closed geodesics on $M$ with multiplicity -- see \cite[Chapter 10]{BorthwickBook}. The zeros of $ Z_M $ enjoy many interpretations, in particular as
{\em quantum resonances} of the Laplacian on $ M $ -- see \cite{revres} 
for a general introduction and references. In particular, finding {\em resonance free} regions has a long tradition and applications in many settings.

The \emph{limit set}, $ \Lambda_\Gamma $, is defined as the set of accumulation points of 
orbits of $ \Gamma$ acting on~$\mathbb H^2$, see also~\eqref{e:limit-set}.
It is a subset of the boundary of $ \mathbb H^2 $ at infinity,
so in the Poincar\'e disk model of $ \mathbb H^2 $ we have
$ \Lambda_\Gamma \subset \mathbb S^1 $. 
The set $\Lambda_\Gamma$ is $\delta$-regular where $\delta\in [0,1)$ is the
exponent of convergence of Poincar\'e series, see~\cite[Lemma~14.13]{BorthwickBook}.

\renewcommand\thefootnote{\dag}%

The \emph{hyperbolic} version of fractal uncertainty principle was
formulated by Dyatlov--Zahl~\cite[Definition 1.1]{hgap}.
Define the operator $\mathcal B_\chi=\mathcal B_\chi(h)$
on~$L^2(\mathbb S^1)$ by
\begin{equation}
  \label{e:B-chi}
\mathcal B_\chi(h)f(y)=(2\pi h)^{-1/2}\int_{\mathbb S^1}|y-y'|^{-2i/h}\chi(y,y')f(y')\,dy',
\end{equation}
where $ | y - y'  |$ is the Euclidean distance between $ y,y' \in 
\mathbb S^1 \subset \mathbb R^2.%
$\footnote{Note the sign change from~\cite[(1.6)]{hgap}. It is convenient for us and it does not change the norm.}
Here
\begin{equation}
\label{eq:cutoff}
\chi\in \CIc(\mathbb S^1_\Delta),\quad
\mathbb S^1_\Delta:=\{(y,y')\in \mathbb S^1\times\mathbb S^1\mid y\neq y'\}.
\end{equation}
We say that $\Lambda_\Gamma$ satisfies \emph{(hyperbolic) fractal uncertainty principle
with exponent $\beta\geq 0$} if for each $\varepsilon>0$ there exists $\rho<1$ such that
for all $C_0>0$, $\chi\in\CIc(\mathbb S^1_\Delta)$ we have
\begin{equation}
  \label{e:fup}
\|\indic_{\Lambda_\Gamma(C_0h^\rho)}\mathcal B_\chi(h)\indic_{\Lambda_\Gamma(C_0h^\rho)}\|_{L^2(\mathbb S^1)\to L^2(\mathbb S^1)}
=\mathcal O(h^{\beta-\varepsilon})\quad\text{as }h\to 0.
\end{equation}
This hyperbolic fractal uncertainty principle with $\beta=\beta(\Gamma)>0$
was established for arbitrary convex co-compact groups by Bourgain--Dyatlov~\cite{fullgap} by proving the flat version~\eqref{eq:fup} and showing that it implies~\eqref{e:fup}. 
It followed earlier partial results of
Dyatlov--Zahl~\cite{hgap}.

With this in place we can now state our main result:

\medskip
\noindent
{\bf Theorem.} 
{\em Assume that $M=\Gamma\backslash\mathbb H^2$ is a convex co-compact hyperbolic surface
and the limit set $\Lambda_\Gamma$ satisfies fractal uncertainty principle with exponent $\beta$
in the sense of~\eqref{e:fup}.
Then $M$ has an \textbf{essential spectral gap}
of size $\beta-$, that is for each $\varepsilon>0$ the Selberg zeta function of $M$
has only finitely many zeroes in $\{\Re s\geq{1\over 2}-\beta+\varepsilon\}$.}
\smallskip

\Remark Bourgain--Dyatlov~\cite{hyperfup} showed that when $\delta>0$, the set $\Lambda_\Gamma$
also satisfies the fractal uncertainty principle with exponent $\beta>{1\over 2}-\delta$
which only depends on~$\delta$. The resulting essential spectral gap is an improvement
over the earlier work of Naud~\cite{NaudGap} which gave a gap of size
$\beta>{1\over 2}-\delta$ which also depends on the surface.
See also Dyatlov--Jin~\cite{regfup}.

A stronger theorem was proved in \cite{hgap} using fine microlocal methods
which included second microlocalization and Vasy's method for meromorphic continuation (see~\cite[\S 3.1]{revres} and references given there). 
In addition to showing a zero 
free strip, it provided a bound on the scattering resolvent, see~\cite[Theorem 3]{hgap}. Having such bounds is essential to applications~-- see~\cite[\S 3.2]{revres}. It would be interesting to see if one can use the methods of this paper to obtain bounds on the scattering resolvent. However, generalizations to non-constant curvature are likely to be based on the microlocal techniques of~\cite{hgap}.
The result of~\cite{hgap} also applies to higher-dimensional hyperbolic quotients.
The proof of the present paper can be adapted to higher dimensional cases where
the limit set is totally discontinuous (such as Schottky quotients).

The present paper instead uses \emph{transfer operators}~-- see the outline below and~\S\ref{s:transfer}. The identification of $M $ with a quotient by a Schottky group~-- see~\S \ref{s:schottky}~-- allows to combine simple semiclassical insights with the combinatorial structure of the Schottky data.

\smallsection{Outline of the proof}
The proof is based on a well known identification of zeros of $ Z_M ( s ) $ with the values of~$ s $ at which 
the {\em Ruelle transfer operator}, $ \mathcal L_s $, has eigenvalue~$ 1 $ -- see \eqref{eq:delLs}. Referring to \S \ref{s:transfer} and
\cite[\S 15.3]{BorthwickBook} for precise definitions, we have 
\begin{equation} \label{eq:LsBS}
\mathcal L_s u ( z ) = \sum_{ w\colon B w = z } B' (  w)^{-s} u (  w ) , \end{equation}
where $ B $ is an expanding {\em Bowen--Series map} and $ u $ is a holomorphic function on a family of disjoint disks symmetric with respect to the real axis.  
\renewcommand\thefootnote{\ddag}%

\begin{figure}
\includegraphics{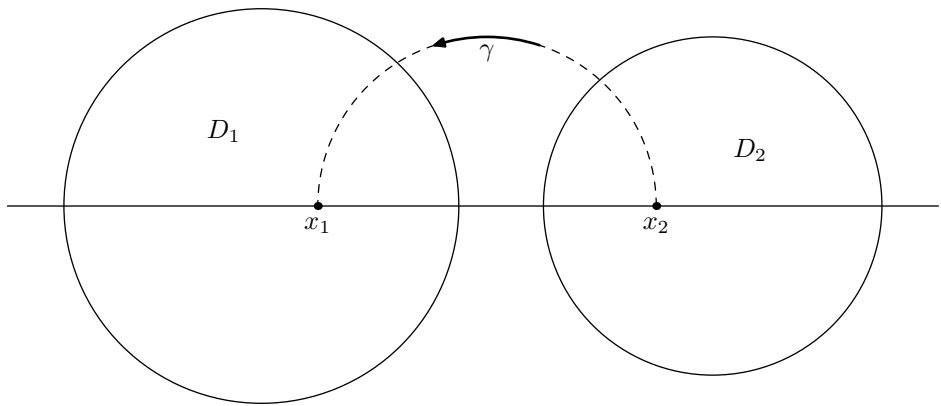}
\caption{The disks $D_1,D_2$ and the transformation $\gamma$
in the trivial case~\eqref{e:outline-display}.}
\label{f:basic-schottky}
\end{figure}
We outline the idea of the proof in the \emph{trivial} case (the zeros of 
$ Z_M ( s ) $ can be computed explicitly -- see~\cite[(10.32)]{BorthwickBook}) when 
there are only two disks $D_1,D_2$ and $ \gamma \in \SL(2,\mathbb R)$,
a linear fractional transformation preserving the upper half plane,
maps one disk onto the complement of the other disk,
see Figure~\ref{f:basic-schottky}:%
 \footnote{This trivial case can be reduced to $ \gamma ( z ) = kz $, $ k > 1 $, but we do not want to stress that point.} 
\begin{equation}
  \label{e:outline-display}
 D_2 = \dot {\mathbb C}
 \setminus \gamma^{-1}(D_1^\circ  ) , \quad 
B(z)=\begin{cases} \gamma^{-1}(z),&z\in D_1;\\
\gamma(z),&z\in D_2.
\end{cases}
\end{equation}
Denote $I_j:=D_j\cap\mathbb R$.
In the case~\eqref{e:outline-display} the limit set $\Lambda_\Gamma$ consists of the fixed points of~$\gamma $,  
\[ 
\Lambda_\Gamma = \{ x_1, x_2 \} \subset \dot {\mathbb R }  , \ \  x_j \in I_j ,\ \ \gamma ( x_j ) = x_j , 
\]
and the fractal uncertainty 
principle \eqref{e:fup} holds with $\beta={1\over 2}$ by a \emph{trivial volume bound}: 
\begin{equation}
  \label{e:intro-fupuse}
\begin{aligned} 
\|\indic_{\Lambda_\Gamma (C_0h^\rho)}\mathcal B_\chi(h)&\indic_{\Lambda_\Gamma(C_0h^\rho)}\|_{L^2 \to L^2}  \\
& \leq \| \indic_{\Lambda_\Gamma (C_0h^\rho)} \|_{ L^\infty \to L^2 } 
\| \mathcal B_\chi(h) \|_{L^1 \to L^\infty } \| \indic_{\Lambda_\Gamma(C_0h^\rho)} \|_{ L^2 \to L^1 } \\
& \leq C h^{\frac \rho 2} \times h^{-\frac12} \times h^{\frac \rho 2 } 
=\mathcal O(h^{\frac12-\varepsilon}), \quad \varepsilon=1-\rho. \ 
 \end{aligned}
\end{equation}
The operator \eqref{eq:LsBS} splits as a direct sum over operators on the two disks and hence we only need to consider a simplified transfer operator
\begin{equation}
  \label{eq:defLsu}
\mathcal L_s u ( z ) = \gamma' ( z )^s u ( \gamma ( z ) ) , \quad u \in \mathcal H ( D_1 ) , 
\end{equation}
where $ \mathcal H ( D_1 ) $ is the Bergman space of holomorphic functions in $ L^2 ( D_1 )$.

Given~\eqref{e:intro-fupuse} we want to show, in a complicated way which generalizes, that the equation $ \mathcal L_s u = u $, $ u \in \mathcal H ( D_1 ) $, has no non-trivial solutions for 
\begin{equation*}
s=\tfrac{1} 2-\nu+  i h^{-1}   \  
\text{  \label{e:s-given}
with $ \ \nu < \tfrac12 - \varepsilon $. }
\end{equation*}
Thus assume that for such an $ s $, $ \mathcal L_s u = u $. The first observation is that $ u|_{ \mathbb R}  $  is semiclassicaly localized to bounded frequencies: that means that for all $ \chi \in C^\infty_{\rm{c}} ( I_1 ) $, 
\begin{equation}
\label{eq:compm} \mathcal F_h (\chi u) ( \xi) = 
\mathcal O (  h^\infty |\xi|^{-\infty}) \quad \text{for } |\xi| \geq  C .
\end{equation}
It is here
that holomorphy of $ u $ is used: using the maximum principle we show that
$ \sup_{D_1} | u | \leq e^{ C/h} \sup_{ I_1} |u | $ and derive the Fourier transform bound from this.
See Lemmas~\ref{l:a-priori} and~\ref{l:u-fourier}.

From now on we work only on the real axis.
In the outline we will use concepts from semiclassical analysis but we stress that the actual proofs in the paper are  self-contained.

To connect the model used here to $ \mathcal B_\chi ( h ) $, which acts on 
$ \mathbb S^1 $, we identify $ \mathbb S^1 $ with the extended real axis
$ \dot { \mathbb R} $ -- see \S \ref{s:integral-operator}. Transplanted to 
$ \dot {\mathbb R} $, $ \mathcal B_\chi ( h ) $ is a semiclassical Fourier integral operator with the phase
$\Phi(x,x')$ defined in~\eqref{eq:phase} and hence 
associated to the canonical transformation
\[
\varkappa : T^* \dot {\mathbb R} \to 
T^* \dot {\mathbb R } , \quad \varkappa :( x', -\partial_{x'}\Phi ) \mapsto ( x , \partial_x\Phi ).
\]
We stress that $\varkappa$ is a global diffeomorphism and
for $(x,\xi)=\varkappa(x',\xi')$ we have
$ x \neq x' $. That means that for the action on 
compactly microlocalized functions,  the singularity removed by the
cutoff~\eqref{eq:cutoff} is irrelevant and we can consider a simpler operator
$ \mathcal B ( s )$ defined by, essentially, removing $ \chi $.
In the circle model it has the form
$$
\mathcal B(s) f(y)=(2\pi h)^{-1/2}\int_{\mathbb S^1} |y-y'|^{-2s} f(y')d y'.
$$
See~\S \ref{s:integral-operator} for details.
This operator has a nice equivariance property
which is particularly simple for the operator~\eqref{eq:defLsu}:
denoting $\langle x\rangle:=\sqrt{1+x^2}$ we have
\begin{equation}
  \label{e:intro-equiv}
\widetilde{\mathcal L}_s\mathcal B ( s ) = \mathcal B ( s )
\widetilde{\mathcal L}_{1-s}\quad\text{where}\quad
\widetilde{\mathcal L}_s:=\langle x\rangle^{2s}\mathcal L_s \langle x\rangle^{-2s}:
L^2(\dot{\mathbb R})\to L^2(\dot{\mathbb R}).
\end{equation}
See Lemma~\ref{l:B-equiv} for the general version. Here we only say that what lies behind this identity is the following formula valid for 
linear fractional transformations $ \varphi $:
\begin{equation}
\label{eq:rho}   | \varphi ( x ) - \varphi ( y ) |^{-2} | \varphi' ( x ) | = | x - y |^{-2} | \varphi' ( y ) |^{-1} .
\end{equation}

To use \eqref{e:intro-equiv}, we put $\gamma_N:=\gamma^N $, the $N$th iterate where $N\sim \log(1/h)$ is chosen so that
$ | \gamma_N ( I_1) | \sim h^\rho $. Then
$$
u(x)=\mathcal L_s^Nu(x)=\gamma_N'(x)^s u(\gamma_N(x)),\quad x\in I_1.
$$
Choose a cutoff function $\chi_N \in C_{\rm{c}}^\infty ( ( x_1- C_1 h^\rho, x_1+C_1 h^\rho ) ) $ which is equal to~1 in a neighbourhood of 
$ \gamma_N (I_1) $. We then  put
$u_N:=\langle x\rangle^{2s}\chi_N u$,
so that 
$\langle x\rangle^{2s}u=\widetilde{\mathcal L}_s^N u_N$ on $I_1$.

Since $\rho<1$, $u_N$ remains compactly microlocalized in the
sense of~\eqref{eq:compm} but in addition it is concentrated at $ x= x_1$:
$$
\WFh(u_N )\subset \{x=x_1,\ |\xi|\leq C\}.
$$
The operator $\mathcal B(s)$ is elliptic and we denote its {\em microlocal inverse} 
by $\mathcal B(s)^{-1}$ noting that it is a Fourier integral operator
associated to $\varkappa^{-1}$. Hence, 
$$
\begin{gathered}
u_N=\mathcal B(s)v_N+\mathcal O(h^\infty),\quad
v_N:=\mathcal B(s)^{-1}u_N, \\
\WFh(v_N)\subset \varkappa^{-1}(\WFh(u_N))\subset \{x\neq x_1\}.
\end{gathered}
$$
Therefore, changing $v_N$ by $\mathcal O(h^\infty)$, we may assume that
$\supp v_N\subset J\subset\dot{\mathbb R}$
where $J$ is a fixed `interval' on $ \dot{\mathbb R} \simeq 
\mathbb S^1 $ and $ x_1 \notin J $.

Now we use the equivariance property~\eqref{e:intro-equiv}: modulo an $\mathcal O(h^\infty)$ error, we have
$$
\begin{gathered}
\langle x\rangle^{2s}u=\widetilde{\mathcal L}_s^Nu_N=\widetilde{\mathcal L}_s^N \mathcal B(s)v_N
=\mathcal B(s)w_N\quad\text{on }I_1,\\
w_N:=\widetilde{\mathcal L}_{1-s}^N v_N, \quad
\supp w_N\subset \widetilde J:=\gamma^{-N}(J) .
\end{gathered}
$$
Since $J$ lies a fixed distance away from $ x_1$, $ \widetilde J = \gamma^{-N} ( J )$ lies in an $h^\rho$ sized interval centered at the repelling point, $ x_2 $, of the transformation $\gamma$. The change of variables in the integrals shows that
$$
\|w_N\|_{L^2}\sim h^{-\rho\nu}\|v_N\|_{L^2}\sim h^{-\rho\nu}\|u_N\|_{L^2}.
$$
Now we have
$$
u_N=\chi_N\mathcal B(s)\indic_{\widetilde J}w_N+\mathcal O(h^\infty)
$$
so the fractal uncertainty principle~\eqref{e:intro-fupuse} gives
$$
\|u_N\|_{L^2}\leq \|\chi_N\mathcal B(s)\indic_{\widetilde J}\|_{L^2\to L^2}
\|w_N\|_{L^2} \leq Ch^{1/2-\varepsilon-\rho\nu}\|u_N\|_{L^2}+\mathcal O(h^\infty).
$$
Since $\nu<{1\over 2}-\varepsilon$ and we can take $ \rho = 1 - \varepsilon $, 
we obtain $u_N\equiv 0$ if $h$ is small enough, thus
$u\equiv 0$ and the proof is finished.

In the case of non-trivial Schottky groups similar ideas work but with combinatorial complications.  
We only make a general comment that for 
iterates $ \gamma_N $ (or more generally iterates $ \gamma_{\mathbf a } $
-- see \S \ref{s:schottky}), $  \partial_x \log | \gamma_{N} ' ( x) | $ for 
$  x $ in a small $h$-independent neighbourhood of~$ x_1 $ is essentially equal to $ - 2 \partial_x \log | x - \gamma_N^{-1} ( x_0 )  | $, $ 
x_0 \neq x_1 $.
This follows from \eqref{eq:rho} with $y:=\gamma_N^{-1}(x_0)$, $\varphi:=\gamma_N$. Any generalization of our method has to replace this explicit formula by 
writing $ \partial_x \log | \gamma_N ' ( x ) | $ approximately as
$ \partial_x \Phi ( x , y) $, $ y = \gamma_N^{-1} ( x_0 ) $, $ x_0 \neq x_1$, with $ \Phi $ generating a canonical tranformation.

\medskip\noindent\textbf{Notation:} 
We write $ f = 
\mathcal O ( g )_H $ for $ \| f \|_H \leq g $. In particular, 
$ f = \mathcal O ( h^\infty )_H  $ means that for any $ N $ there exists
$ C_N $ such that $ \|f \|_H \leq C_N h^N $. The norm 
$ \| \bullet \| $ refers to the $ L^2 $ norm but different $ L^2 $ norms are used ($ L^2 ( \mathbb R ) $, $ L^2 ( \mathbb S^1 ) $, $ L^2 ( \Omega ) $) and they are either specified or clear from the context.
With some of abuse of notation, $C_\Gamma$ denotes large constants
which only depend on a fixed Schottky data of $ \Gamma $ (see \S \ref{s:schottky}) and whose exact value may vary from place to place.
We denote $\langle x\rangle:=\sqrt{1+x^2}$.

\medskip\noindent\textbf{Acknowledgements.}
We would like to thank the anonymous referee for suggestions to improve the manuscript.
This research was conducted during the period SD served as
a Clay Research Fellow. MZ was supported
by the National Science Foundation grant DMS-1500852 and by a Simons Fellowship. 

\section{Ingredients}

\subsection{Schottky groups}
  \label{s:schottky}

In this section we briefly review properties of Schottky groups,
referring the reader to~\cite[\S15.1]{BorthwickBook} for more details.
We use notation similar to~\cite[\S2.1]{hyperfup}.

The group $\SL(2,\mathbb R)$ acts on the extended complex plane
$\dot{\mathbb C}:=\mathbb C\cup\{\infty\}$ by M\"obius transformations:
$$
\gamma=\begin{pmatrix} a & b \\ c & d \end{pmatrix}\in \SL(2,\mathbb R),\
z\in\dot{\mathbb C}\quad \Longrightarrow\quad
\gamma(z)={az+b\over cz+d}.
$$
Denote by $\mathbb H^2\subset\mathbb C$ the upper half-plane model of the hyperbolic plane.
Then $\SL(2,\mathbb R)$ acts on $\mathbb H^2$ by isometries.

A \emph{Schottky group} is a convex co-compact subgroup $\Gamma\subset\SL(2,\mathbb R)$
constructed in the following way (see Figure~\ref{f:schottky}):
\begin{figure}
\includegraphics{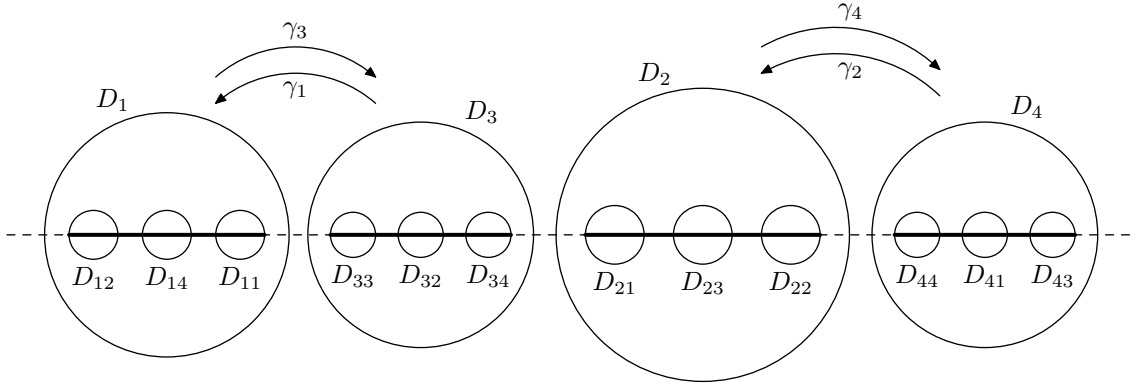}
\caption{A Schottky structure with $r=2$. The real line is dashed.
The solid intervals are $I'_1,I'_3,I'_2,I'_4$.}
\label{f:schottky}
\end{figure}
\begin{itemize}
\item Fix $r\in\mathbb N$ and nonintersecting closed disks $D_1,\dots,D_{2r}\subset\mathbb C$
centered on the real line.
\item Define the \emph{alphabet} $\mathcal A=\{1,\dots,2r\}$ and for each $a\in\mathcal A$,
denote
$$
\overline a:=\begin{cases}
a+r,& 1\leq a\leq r;\\
a-r,& r+1\leq a\leq 2r.
\end{cases}
$$
\item Fix group elements $\gamma_1,\dots,\gamma_{2r}\in \SL(2,\mathbb R)$ such that for all $a\in\mathcal A$
\begin{equation}
  \label{e:gamma-a-property}
\gamma_a(\dot{\mathbb C}\setminus D_{\overline a}^\circ)=D_a,\quad
\gamma_{\overline a}=\gamma_a^{-1}.
\end{equation}
(In the notation of \cite[\S15.1]{BorthwickBook} we have $ \gamma_a = S_a^{-1}$.)
\item Let $\Gamma\subset\SL(2,\mathbb R)$ be the  group generated by $\gamma_1,\dots,\gamma_r$; it is a free group on $ r $ generators 
\cite[Lemma 15.2]{BorthwickBook}.
\end{itemize}

Every convex co-compact hyperbolic surface $M$ can be viewed as the quotient
of~$\mathbb H^2$ by a Schottky group $\Gamma$,
see for instance~\cite[Theorem~15.3]{BorthwickBook}.
Note that the complement of $\bigsqcup_a D_a$ in $\mathbb H^2$
is a fundamental domain.
We  fix a Schottky representation for a given hyperbolic surface $M$ 
and refer to $ D_1 , \dots, D_{2r} $ and $ \gamma_1 , \dots, \gamma_{2r} $ as {\em Schottky data}.

In several places we will use results of~\cite[\S 2]{hyperfup} which 
can be read independently of the rest of~\cite{hyperfup}. 
In particular, 
we use combinatorial notation for indexing the words
in the free group $\Gamma$ and the corresponding disks:
\begin{itemize}
\item For $n\in\mathbb N_0$, define $\mathcal W_n$, the set of words of length~$n$, as follows:
\begin{equation}
\label{eq:Wn}
\mathcal W_n:=\{a_1\dots a_n\mid a_1,\dots,a_n\in\mathcal A,\quad
a_{j+1}\neq\overline{a_j}\quad\text{for }j=1,\dots,n-1\}.
\end{equation}
\item Denote by $\mathcal W:=\bigcup_n\mathcal W_n$ the set of all words. For
$\mathbf a\in\mathcal W_n$, put $|\mathbf a|:=n$. Denote the empty word by $\emptyset$
and put $\mathcal W^\circ:=\mathcal W\setminus \{\emptyset\}$.
\item For $\mathbf a=a_1\dots a_n\in\mathcal W$, put
$\overline{\mathbf a}:=\overline{a_n}\dots\overline{a_1}\in \mathcal W$.
For a set $Z\subset \mathcal W$, put
\begin{equation}
  \label{e:Z-bar}
\overline Z:=\{\overline{\mathbf a}\mid \mathbf a\in Z\}.
\end{equation}
\item For $\mathbf a=a_1\dots a_n\in\mathcal W^\circ$, put $\mathbf a':=a_1\dots a_{n-1}\in\mathcal W$.
Note that $\mathcal W$ forms a tree with root $\emptyset$ and each $\mathbf a\in\mathcal W^\circ$ having parent
$\mathbf a'$.
\item For $\mathbf a=a_1\dots a_n,\mathbf b=b_1\dots b_m\in\mathcal W$, we write $\mathbf a\to \mathbf b$
if either $\mathbf a=\emptyset$, or $\mathbf b=\emptyset$, or $a_n\neq\overline{b_1}$. Under this condition
the concatenation $\mathbf{ab}$ is a word.
We write $\mathbf a\rsa\mathbf b$ if $\mathbf a,\mathbf b\in\mathcal W^\circ$
and $a_n=b_1$. In the latter case $\mathbf a'\mathbf b\in\mathcal W$.
\item For $\mathbf a,\mathbf b\in\mathcal W$, we write $\mathbf a\prec\mathbf b$
if $\mathbf a$ is a prefix of $\mathbf b$, i.e. $\mathbf b=\mathbf{ac}$
for some $\mathbf c\in\mathcal W$.
\item Define the following one-to-one correspondence between $\mathcal W$ and the group $\Gamma$:
$$
\mathcal W \ni a_1\dots a_n = \mathbf a \quad\longmapsto\quad
\gamma_{\mathbf a}:=\gamma_{a_1}\cdots\gamma_{a_n}\in\Gamma.
$$
Note that
$\gamma_{\mathbf a\mathbf b}=\gamma_{\mathbf a}\gamma_{\mathbf b}$ when $\mathbf a\to\mathbf b$,
$\gamma_{\overline{\mathbf a}}=\gamma_{\mathbf a}^{-1}$,
and $\gamma_\emptyset$ is the identity.
\item For $\mathbf a=a_1\dots a_n\in\mathcal W^\circ$, define the disk centered on the real line
(see Figure~\ref{f:schottky})
$$
D_{\mathbf a}:=\gamma_{\mathbf a'}(D_{a_n})\subset\mathbb C.
$$
If $\mathbf a\prec\mathbf b$, then $D_{\mathbf b}\subset D_{\mathbf a}$. On the
other hand, if $\mathbf a\not\prec\mathbf b$ and $\mathbf b\not\prec\mathbf a$,
then $D_{\mathbf a}\cap D_{\mathbf b}=\emptyset$.
Define the interval 
$$
I_{\mathbf a}:=D_{\mathbf a}\cap\mathbb R
$$
and denote by $|I_{\mathbf a}|$ its length (which is equal to the diameter
of $D_{\mathbf a}$).
\item For $a\in\mathcal A$, define the interval $I'_a\subset I_a^\circ$ as the convex
hull of the union $\bigsqcup_{b\in\mathcal A,\,a\to b}I_{ab}$, see Figure~\ref{f:schottky}.
More generally, for $\mathbf a=a_1\dots a_n\in \mathcal W^\circ$ define
\begin{equation}
  \label{e:I-a-prime}
I'_{\mathbf a}:=\gamma_{\mathbf a'}(I_{a_n}')\subset I_{\mathbf a}^\circ.
\end{equation}
Note that $I'_{\mathbf a}\supset I_{\mathbf b}$ for any $\mathbf b\in\mathcal W^\circ$ such that
$\mathbf a\prec\mathbf b$, $\mathbf a\neq \mathbf b$.
\item Denote
\begin{equation}
  \label{e:D-def}
\mathbf D:=\bigsqcup_{a\in\mathcal A}D_a\subset \mathbb C,\quad
\mathbf I:=\bigsqcup_{a\in\mathcal A}I_a=\mathbf D\cap\mathbb R,\quad
\mathbf I':=\bigsqcup_{a\in\mathcal A}I'_a\subset \mathbf I^\circ.
\end{equation}
\item The \emph{limit set} is given by
\begin{equation}
  \label{e:limit-set}
\Lambda_\Gamma:=\bigcap_{n\geq 1}\bigsqcup_{\mathbf a\in\mathcal W_n}D_{\mathbf a}\subset\mathbb R.
\end{equation}
The fact that $\Lambda_\Gamma\subset \mathbb R$ follows from the
\emph{contraction property}~\cite[\S2.1]{hyperfup}
\begin{equation}
  \label{e:contracting}
|I_{\mathbf a}|\leq C_\Gamma(1-C_\Gamma^{-1})^{|\mathbf a|}\quad\text{for all }\mathbf a\in\mathcal W^\circ.
\end{equation}
\end{itemize}
We finish this section with a few estimates.
We start with the following derivative bound which is the complex version of~\cite[Lemma~2.5]{hyperfup}:
\begin{lemm}
  \label{l:gamma-der}
For all $\mathbf a=a_1\dots a_n\in \mathcal W^\circ$ and $z\in D_{a_n}$ we have
\begin{equation}  
  \label{e:gamma-der}
C_\Gamma^{-1}|I_{\mathbf a}|\leq |\gamma_{\mathbf a'}'(z)|
\leq C_\Gamma |I_{\mathbf a}|.
\end{equation}
\end{lemm}
\begin{proof}
Define the intervals $I:=I_{a_n}$, $J:=I_{\mathbf a}=\gamma_{\mathbf a'}(I)$.
Let $\gamma_I,\gamma_J\in \SL(2,\mathbb R)$ be the unique affine transformations
such that $\gamma_I([0,1])=I$, $\gamma_J([0,1])=J$.
Following~\cite[\S2.2]{hyperfup}, we write
$$
\gamma_{\mathbf a'}=\gamma_J\gamma_\alpha\gamma_I^{-1},\quad
\gamma_\alpha=\begin{pmatrix} e^{\alpha/2} & 0 \\
e^{\alpha/2} - e^{-\alpha / 2} & e^{-\alpha / 2}\end{pmatrix} \in \SL(2,\mathbb R),
$$
where $\alpha\in\mathbb R$ and $|\alpha|\leq C_\Gamma$ by~\cite[Lemma~2.4]{hyperfup}.
For $z\in D_{a_n}$ we have
$$
\gamma_{\mathbf a'}(z)={|J|\over |I|}\gamma_\alpha'(w),\qquad
w=\gamma_I^{-1}(z)\in \Big\{\Big|w-{1\over 2}\Big|\leq {1\over 2}\Big\}.
$$
We compute
$$
|\gamma_{\alpha}'(w)|={e^\alpha\over |(e^{\alpha}-1)w+1|^2}\in [C_\Gamma^{-1},C_\Gamma] .
$$
Since  
$ C_\Gamma^{-1}  | I_{\mathbf a} | \leq { |J|}/{|I|} \leq C_\Gamma
| I_{\mathbf a } | $, 
\eqref{e:gamma-der} follows.
\end{proof}
The next lemma bounds the number of intervals $I_{\mathbf a}$ of comparable sizes
which can contain a given point:
\begin{lemm}
  \label{l:multiplicity}
For all $C_1\geq 2$ and $\tau>0$, we have
\begin{equation}
  \label{e:multiplicity}
\sup_{x\in\mathbb R} \#\{\mathbf a\in \mathcal W^\circ\colon 
\tau\leq |I_{\mathbf a}|\leq C_1\tau,\ x\in I_{\mathbf a}\}
\leq C_\Gamma\log C_1.
\end{equation}
\end{lemm}
\begin{proof}
Fix $x\in \mathbb R$. We have
$$
\{\mathbf a\in \mathcal W^\circ\colon 
\tau\leq |I_{\mathbf a}|\leq C_1\tau,\ x\in I_{\mathbf a}\}
=\{\mathbf a_1,\dots,\mathbf a_N\}
$$
for some $\mathbf a_1,\dots,\mathbf a_N\in\mathcal W^\circ$
such that $\mathbf a_1\prec\mathbf a_2\prec\dots\prec\mathbf a_N$.
We have $|I_{\mathbf a_{j+1}}|\leq (1-C_{\Gamma}^{-1})|I_{\mathbf a_j}|$,
see~\cite[\S2.1]{hyperfup}, and~\eqref{e:multiplicity} follows.
\end{proof}

\subsection{Functional spaces}
\label{s:funsp}

For any open $\Omega\subset\mathbb C$, let $\mathcal H(\Omega)$ be the \emph{Bergman space} on~$\Omega$,
consisting of holomorphic functions $f:\Omega\to\mathbb C$ such that
$f\in L^2(\Omega)$ (with respect to the Lebesgue measure).
Endowing $\mathcal H(\Omega)$ with the $L^2$ norm, we obtain a separable Hilbert space.
For general $\Omega\subset\mathbb C$, we put $\mathcal H(\Omega):=\mathcal H(\Omega^\circ)$.

Denote
\begin{equation}
  \label{e:D-2}
\mathbf D_2:=\bigsqcup_{\mathbf a\in \mathcal W\atop |\mathbf a|=2} D_{\mathbf a}\subset \mathbf D^\circ
\end{equation}
and let $\mathbf D^\pm:=\mathbf D\cap \{\pm\Im z\geq 0\}$,
$\mathbf D^\pm_2:=\mathbf D_2\cap \{\pm \Im z\geq 0\}$. See Figure~\ref{f:D-2}.

The following basic estimate is needed for the a priori bounds in~\S\ref{s:a-priori} below:
\begin{lemm}
  \label{l:interpolated-bound}
There exists $c\in (0,1]$ such that
for all $ f \in \mathcal H ( \mathbf D^\pm ) \cap C ( \mathbf D^{\pm } ) $,
\begin{equation}
  \label{e:ib}
\sup_{\mathbf D_2^\pm}|f|\leq \big(\sup_{\mathbf I}|f|\big)^{c}
\big(\sup_{\mathbf D^\pm}|f|\big)^{1-c}.
\end{equation}
\end{lemm}
\begin{figure}
\includegraphics{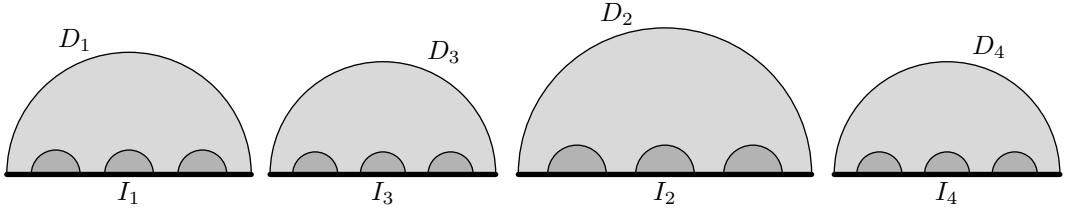}
\caption{An illustration of Lemma~\ref{l:interpolated-bound}. The lighter
shaded region is~$\mathbf D^+$ and the darker region is $\mathbf D_2^+$.
The union of the thick lines is~$\mathbf I$.}
\label{f:D-2}
\end{figure}
\begin{proof}
The boundary of $\mathbf D^\pm$ consists of $\mathbf I$ and a union of half-circles,
which we denote $\mathbf S^\pm$.
Let $F_\pm:(\mathbf D^\pm)^\circ\to [0,1]$ be the harmonic
function with boundary values
$$
F_\pm|_{\mathbf I}\equiv 1,\quad
F_\pm|_{\mathbf S^\pm}\equiv 0.
$$
Since $ F_\pm $ is positive in $(\mathbf D^\pm)^\circ $ and
$\mathbf D_2^\pm$ lies away from $\mathbf S^\pm$,
the infimum of $F_\pm$ on $ \mathbf D_2^\pm $ is positive.  Denote this infimum by 
$
c:=\inf_{\mathbf D_2^\pm}F^\pm\in (0,1]
$.
Since $\log|f|$ is subharmonic in $\mathbf D^\pm$ and
$$
\log|f|\leq \big(\sup_{\mathbf I}\log|f|\big)F_\pm+\big(\sup_{\mathbf D^\pm}\log|f|\big)(1-F_\pm)\quad\text{on }\partial\mathbf D^\pm,
$$
the maximum principle implies that for $ z \in \mathbf D^\pm_2 $, 
\[
\begin{split} 
\log |f ( z)|  & \leq 
\sup_{\mathbf D^\pm}\log|f| - 
F_\pm ( z ) \big( \sup_{\mathbf D^\pm}\log|f| - 
\sup_{\mathbf I}\log|f|  \big) \\
& \leq \sup_{\mathbf D^\pm}\log|f|  - c 
\big( \sup_{\mathbf D^\pm}\log|f| - 
\sup_{\mathbf I}\log|f|  \big) .
\end{split} 
\]
Exponentiating this we obtain~\eqref{e:ib}.
\end{proof}

\subsection{Transfer operators and resonances}
  \label{s:transfer}
The zeros of the Selberg zeta function, that is the 
resonances of $ M = \Gamma \backslash \mathbb H^2 $, are characterized
using a dynamical transfer operator, also called the {\em Ruelle operator}.
Here we follow \cite[\S 15.3]{BorthwickBook},\cite{GLZ} and consider these operators on Bergman 
spaces defined in \S \ref{s:funsp}. We refer to \cite{Ballade} for 
other approaches to transfer operators and for historical background.

Here, for $s\in\mathbb C$ we define the \emph{transfer operator} $\mathcal L_s:\mathcal H(\mathbf D)\to\mathcal H(\mathbf D)$ 
as follows:
\begin{equation}
  \label{e:transfer}
\mathcal L_s f(z)=\sum_{a\in\mathcal A\atop a\to b} \gamma_a'(z)^s f(\gamma_a(z)),\quad
z\in D_b,\
b\in\mathcal A.
\end{equation}
(We note that in the notation of \cite[(15.11)]{BorthwickBook}, $ S_i =
\gamma_i^{-1} $.) This is the same as \eqref{eq:LsBS} if we define 
$ B z = \gamma_a^{-1} ( z ) $ for $ z \in D_a $.

The derivative satisfies $\gamma_a'(z)>0$ for $z\in I_b$ and  $\gamma_a'(z)^s$
is uniquely defined and holomorphic for $z\in D_b$ and $s\in\mathbb C$ such that
$\gamma_a'(z)^s>0$ when $z\in I_b$, $s\in\mathbb R$. Since
$\gamma_a$ takes real values on $\mathbb R$, the expression~\eqref{e:transfer} additionally
gives an operator on $L^2(\mathbf I)$, also denoted $\mathcal L_s$.
The operators $\mathcal L_s$ are of trace class 
(see \cite[Lemma 15.7]{BorthwickBook}) and depend holomorphically on $s\in\mathbb C$.
Hence the determinant
$
\det(I-\mathcal L_s)
$
defines an entire function of $s\in\mathbb C$.
The connection to the Selberg zeta function defined in \eqref{eq:Selb} 
is a special case of Ruelle theory and 
and can be found in \cite[Theorem 15.10]{BorthwickBook}:
\[ Z_M ( s) = \det ( I - \mathcal L_s ) . \]
It also follows that $I-\mathcal L_s$ is a Fredholm operator of index 0 and 
its invertibility is equivalent to $ \ker ( I - \mathcal L_s ) = \{ 0 \}$.
Hence,
\begin{equation}
\label{eq:delLs}  Z_M ( s ) = \det ( I - \mathcal L_s ) = 0 
\ \Longleftrightarrow \ \exists \,  u \in \mathcal H ( \mathbf D ) , \ 
u \not \equiv 0 , \ u = \mathcal L_s u ,
\end{equation}
see \cite[Theorem A.34]{BorthwickBook}.

\subsection{Partitions and refined transfer operators}
\label{s:parif}

Our proof uses refined transfer operators which are generalizations of powers of the
standard transfer operator $\mathcal L_s$ given by~\eqref{e:transfer}. To introduce these
we use the notion of a partition:
\begin{itemize}
\item A finite set $Z\subset\mathcal W^\circ$ is called a \emph{partition}
if there exists $N$ such that for each $\mathbf a\in\mathcal W$
with $|\mathbf a|\geq N$, there exists unique $\mathbf b\in Z$ such that
$\mathbf b\prec\mathbf a$. See Figure~\ref{f:partition}. In terms of the limit set, this means that
\begin{equation}
\label{eq:nelim}
\Lambda_\Gamma=\bigsqcup_{\mathbf b\in Z} (I_{\mathbf b}\cap\Lambda_\Gamma).
\end{equation}
\item The alphabet $\mathcal A$ is a partition, as is the set $\mathcal W_n$
of words of length $n\geq 1$. Another important example
is the set of words discretizing to some resolution
$\tau>0$:
\begin{equation}
  \label{e:Z-tau}
Z(\tau):=\big\{\mathbf a\in\mathcal W^\circ\colon |I_{\mathbf a}|\leq\tau<|I_{\mathbf a'}|\big\}
\end{equation}
where we put $|I_{\emptyset}|:=\infty$. The set $Z(\tau)$ is a partition due to~\eqref{e:contracting}.
\item Note however that if $Z$ is a partition, this does not imply
that the set $\overline Z$ defined in~\eqref{e:Z-bar} is a partition.
\end{itemize}
If $Z\subset\mathcal W^\circ$ is a finite set, we define the refined transfer operator $\mathcal L_{Z,s}:\mathcal H(\mathbf D)\to\mathcal H(\mathbf D)$
as follows:
\begin{equation}
  \label{e:refined-transfer}
\mathcal L_{Z,s}f(z)=\sum_{\mathbf a\in Z\atop \mathbf a\rsa b}\gamma_{\mathbf a'}'(z)^s f(\gamma_{\mathbf a'}(z)),\quad
z\in D_b,\
b\in\mathcal A.
\end{equation}
As in the case of $\mathcal L_{s}$, we can also consider $\mathcal L_{Z,s}$ as an operator on $L^2(\mathbf I)$.

Here are some basic examples of refined transfer operators:
\begin{itemize}
\item If $Z=\mathcal A$ then $\mathcal L_{Z,s}$ is the identity operator.
\item If $Z=\mathcal W_2$, the set of words of length 2, then
$\mathcal L_{Z,s}=\mathcal L_s$, the standard transfer operator defined in~\eqref{e:transfer}.
\item More generally if $Z=\mathcal W_N$ for some $N\geq 1$, then
$\mathcal L_{Z,s}=\mathcal L_s^{N-1}$.
\end{itemize}
%
\begin{lemm}
  \label{l:transfer-power}
Assume that $Z$ is a partition; define $\overline Z$ by~\eqref{e:Z-bar}.
Then for all $u\in \mathcal H(\mathbf D)$ and $s\in\mathbb C$
\begin{equation}
  \label{e:transfer-power}
\mathcal L_s u=u\quad\Longrightarrow\quad \mathcal L_{\overline Z,s}u=u.
\end{equation}
\end{lemm}
\begin{proof}
We argue by induction on $\sum_{\mathbf b\in Z}|\mathbf b|$. If $Z=\mathcal A$ then
$\mathcal L_{\overline Z,s}$ is the identity operator so~\eqref{e:transfer-power} holds.
Assume that $Z\neq\mathcal A$. Choose a longest word $\mathbf dc\in Z$, where $\mathbf d\in \mathcal W^\circ$
and $c\in \mathcal A$. Then $Z$ has the form (see Figure~\ref{f:partition})
\begin{figure}
\includegraphics{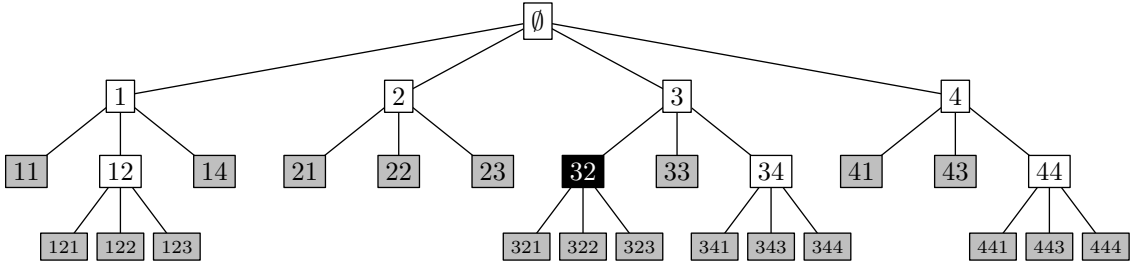}
\caption{An example of a partition $Z$, with elements of $Z$ shaded grey in the tree of words.
The solid black word is one possible choice of $\mathbf d$
in the proof of Lemma~\ref{l:transfer-power}.}
\label{f:partition}
\end{figure}
$$
Z=\big(Z'\setminus \{\mathbf d\}\big)\sqcup \{\mathbf da\mid a\in\mathcal A,\ \mathbf d\to a\}
$$
where $Z'$ is a partition containing $\mathbf d$. By the inductive hypothesis
we have $\mathcal L_{\overline{Z'},s}u=u$, thus is remains to prove
that
\begin{equation}
  \label{e:tp-1}
\mathcal L_s u=u\quad\Longrightarrow\quad \mathcal L_{\overline{Z'},s}u=\mathcal L_{\overline Z,s}u.
\end{equation}
We write
for each $z\in D_b$, $b\in\mathcal A$
$$
\mathcal L_{\overline{Z},s}u(z)=\sum_{\mathbf a\in Z\atop \overline{\mathbf a}\rsa b}
\gamma_{\overline{\mathbf a}'}'(z)^su(\gamma_{\overline{\mathbf a}'}(z))
$$
and similarly for $Z'$. The condition $\overline{\mathbf a}\rsa b$
simply means $b=\overline{a_1}$ where $a_1$ is the first letter of $\mathbf a$.

For $b\neq \overline{d_1}$ the expressions for
$\mathcal L_{\overline{Z},s}u(z)$ and $\mathcal L_{\overline{Z'},s}u(z)$ are identical.
Assume now that $b=\overline{d_1}$. Removing identical terms from
the conclusion of~\eqref{e:tp-1} and using that $\overline{\mathbf da}'=\overline a\overline{\mathbf d}'$ we
reduce~\eqref{e:tp-1} to
$$
\gamma'_{\overline{\mathbf d}'}(z)^su(\gamma_{\overline{\mathbf d}'}(z))
=\sum_{a\in\mathcal A\atop \mathbf d\to a}\gamma'_{\overline a\overline{\mathbf d}'}(z)^s
u(\gamma_{\overline a\overline{\mathbf d}'}(z))\quad\text{for all }z\in D_{\overline{d_1}}.
$$
Using the chain rule and dividing by $\gamma'_{\overline{\mathbf d}'}(z)^s$ this reduces to
$$
u(\gamma_{\overline{\mathbf d}'}(z))
=\sum_{a\in\mathcal A\atop \overline a\to\overline{\mathbf d}}\gamma'_{\overline a}(\gamma_{\overline{\mathbf d}'}(z))^s
u(\gamma_{\overline a}(\gamma_{\overline{\mathbf d}'}(z)))\quad\text{for all }z\in D_{\overline{d_1}}.
$$
The latter follows from the equality $\mathcal L_su(w)=u(w)$
where $w:=\gamma_{\overline{\mathbf d}'}(z)\in D_{\overline{\mathbf d}}$.
\end{proof}

\subsection{An integral operator}
  \label{s:integral-operator}

We now introduce an integral operator $\mathcal B(s)$ similar to $\mathcal B_\chi(h)$ from~\eqref{e:B-chi}. That operator has a simpler definition than 
$ \mathcal B_\chi ( h ) $ but one pays by introducing singularities. 

In our approach we use the upper half plane model while $\mathcal B_\chi(h)$ acts on functions on a circle rather than a line.
Hence, we will use the extended real line $\dot{\mathbb R}=\mathbb R\cup\{\infty\}$
which is identified with the circle $\mathbb S^1\subset\mathbb C$ by the map
\begin{equation}
  \label{e:line-circle}
x\in\dot{\mathbb R}\ \mapsto\ 
y={i-x\over i+x}\in\mathbb S^1.
\end{equation}
The standard volume form on $\mathbb S^1$, pulled back by~\eqref{e:line-circle}, is
$$
dP(x)=2 \langle x\rangle^{-2}\,dx.
$$
Denote $L^2(\dot{\mathbb R}):=L^2(\dot{\mathbb R},dP)\simeq L^2(\mathbb S^1)$.
For $x,x'\in\dot{\mathbb R}$, let $|x-x'|_{\mathbb S}$ be the Euclidean distance
between $y(x)$ and $y(x')$, namely
$$
|x-x'|_{\mathbb S}={2|x-x'|\over \langle x\rangle\langle x'\rangle}.
$$
With this notation in place we define the  operator $\mathcal B(s)$, depending on $s\in\mathbb C$:
for $\Re s<{1\over 2}$, it is a bounded operator on $L^2(\dot{\mathbb R})$ given by the formula
\begin{equation}
  \label{e:B-s}
\mathcal B(s)f(x)=\Big|{\Im s\over 2\pi }\Big|^{1/2}\int_{\dot{\mathbb R}} |x-x'|_{\mathbb S}^{-2s}f(x')\,dP(x').
\end{equation}
For general $s$ the integral in~\eqref{e:B-s} may diverge however 
\begin{equation}
\label{eq:Bschi} \chi_1 \mathcal B(s)\chi_2:L^2(\dot{\mathbb R})\to L^2(\dot{\mathbb R}), 
\ \ \ 
\chi_1,\chi_2\in C^\infty(\dot{\mathbb R}), \ \  \supp\chi_1\cap\supp\chi_2=\emptyset, \end{equation}
is well defined. In other words, 
 $\mathcal B(s)f$ can be defined outside of $\supp f$.

The following equivariance property of $ \mathcal B ( s ) $ will be used in the proof of 
the main theorem in~\S\ref{s:proof-end}:
\begin{lemm}
  \label{l:B-equiv}
Let $\gamma\in \SL(2,\mathbb R)$, $s\in\mathbb C$, and consider the operator
\begin{equation}
  \label{e:T-gamma}
T_{\gamma,s}:L^2(\dot{\mathbb R})\to L^2(\dot{\mathbb R}),\quad
T_{\gamma,s}f(x) :=|\gamma'(x)|^s_{\mathbb S}f(\gamma(x)).
\end{equation}
Here $|\gamma'(x)|_{\mathbb S}=\langle x\rangle^2\langle\gamma(x)\rangle^{-2}\gamma'(x)$
is the derivative of the action of $\gamma$ on the circle defined using~\eqref{e:line-circle}.
Then 
\begin{equation}
  \label{e:B-equivariance}
T_{\gamma,s}\mathcal B(s)=\mathcal B(s)T_{\gamma,1-s}.
\end{equation}
\end{lemm}
\begin{proof}
Take $f\in L^2(\dot{\mathbb R})$. We need to show that
$$
T_{\gamma,s}\mathcal B(s)f(x)=\mathcal B(s)T_{\gamma,1-s}f(x)
$$
for all $x$ when $\Re s<{1\over 2}$ and for all $x\notin \gamma^{-1}(\supp f)$ otherwise.
This is equivalent to
$$
\int_{\dot{\mathbb R}}|\gamma'(x)|_{\mathbb S}^s\cdot|\gamma(x)-x'|_{\mathbb S}^{-2s}f(x')\,dP(x')
=\int_{\dot{\mathbb R}}|x-x''|_{\mathbb S}^{-2s}\cdot|\gamma'(x'')|_{\mathbb S}^{1-s} f(\gamma(x''))\,dP(x'').
$$
The latter follows by the change of variables $x'=\gamma(x'')$ using the identity
$$
|\gamma(x)-\gamma(x'')|_{\mathbb S}^2=|x-x''|_{\mathbb S}^2\cdot |\gamma'(x)|_{\mathbb S}\cdot |\gamma'(x'')|_{\mathbb S}.\qedhere
$$
\end{proof}
\medskip

We now discuss the properties of $\mathcal B(s)$ in the semiclassical limit
which means that we put $ s := \frac12 - \nu + i h^{-1}$, 
where $\nu$ is bounded and $0<h\ll 1$.
In the notation of \eqref{eq:Bschi} we obtain an oscillatory integral 
representation:
$$
\chi_1\mathcal B(s)\chi_2 f(x)=(2\pi h)^{-1/2}\int_{\dot{\mathbb R}}e^{{i\over h}\Phi(x,x')}
\chi_1(x)\chi_2(x')|x-x'|^{2\nu-1}_{\mathbb S}f(x')\,dP(x')
$$
where the phase function $\Phi$ is defined by
\begin{equation}
\label{eq:phase}
\Phi(x,x'):=-2\log |x-x'|_{\mathbb S}=-2\log|x-x'|+2\log\langle x\rangle+2\log\langle x'\rangle
-\log 4.
\end{equation}
Thus $\chi_1\mathcal B(s)\chi_2$ is a semiclassical Fourier integral operator,
see for instance~\cite[\S2.2]{hgap}. 

The next two lemmas can be derived from the theory of these operators but
we present self-contained 
proofs in the Appendix, applying the method of stationary phase directly. That first gives
\begin{lemm}[Boundedness of $\mathcal B(s)$]
  \label{l:B-s-bdd}
Let $\chi_1,\chi_2\in C^\infty(\dot{\mathbb R})$ satisfy
$\supp\chi_1\cap\supp\chi_2=\emptyset$. Then
there exists $C$ depending only on $\nu,\chi_1,\chi_2$ such that
$$
\|\chi_1\mathcal B(s)\chi_2\|_{L^2(\dot{\mathbb R})\to L^2(\dot{\mathbb R})}\leq C.
$$
\end{lemm}

The next lemma gives partial invertibility of $\mathcal B(s)$. To state it
we recall the definition of {\em semiclassical Fourier transform},
\begin{equation}
\label{eq:semif}
\mathcal F_h f ( \xi ) := \int_\mathbb R f ( x ) e^{ - \frac i h x \xi } dx , 
\end{equation}
see \cite[\S 3.3]{ev-zw} for basic properties. We say that
an $h$-dependent family of functions $ f = f(h) $ is 
{\em semiclassically localized to frequencies $ | \xi | \leq M $} if 
for every $ N$, 
\begin{equation}
  \label{eq:semil}
|\mathcal F_h f({\xi}) |\leq C_{N} h^N|\xi|^{-N} \quad\text{when }
 \ |\xi|\geq {M}.
\end{equation}
We also recall {\em semiclassical quantization} 
$ a \mapsto \Op_h ( a ) = a ( x, h D_x ) $, $D_x:={1\over i}\partial_x$
\begin{equation}
\label{eq:qu} \Op_h(a) u ( x ) := \frac{1}{  2 \pi h } \int_{\mathbb R }
e^{\frac i h ( x - y ) \xi } a ( x , \xi )  u ( y )\, dy d \xi, 
\end{equation}
stressing that only elementary properties from~\cite[\S\S 4.2,4.3]{ev-zw} will be used.

The partial invertibility of $ \mathcal B( s ) $ means
that for $f\in L^2(\mathbb R)$ which is supported
in an interval $I$ and semiclassically localized to frequencies
$|\xi|\leq (5|I|)^{-1}$, we have
$f=\mathcal B(s)g+\mathcal O(h^\infty)$ on $I$ for some $g$ which is supported away from $I$: 
\begin{lemm}[Partial invertibility of $\mathcal B(s)$]
  \label{l:B-invert}
Let $I\subset \mathbb R$ be an interval and $K\geq 10$ satisfy $10K|I|\leq 1$.
Assume that
$$
A=\Op_h(a),\quad
a(x,\xi)\in \CIc(\mathbb R^2),\quad
\supp a\subset \{x\in I^\circ,\ |\xi|<2K\}.
$$
Then for all $ \psi_I \in C^\infty_{\rm{c}} ( I^\circ ) $ and $ \omega_I \in 
\CIc ( \dot {\mathbb R} \setminus I ) $ satisfying
\begin{equation}
  \label{e:psi-chosen}
\supp a\subset \{\psi_I(x)=1\},\quad
\supp(1-\omega_I)\subset I+\big[-\tfrac{1}{10K},\tfrac{1}{10K}\big],\quad
\end{equation}
there exists an operator $Q_I (s):L^2(\mathbb R)\to L^2(\dot{\mathbb R})$ uniformly
bounded in $h$, such that
$$
A=\psi_I\mathcal B(s)\omega_I Q_I (s)+\mathcal O(h^\infty)_{L^2(\mathbb R)\to L^2(\mathbb R)}.
$$
\end{lemm}
We recall that the proofs of both lemmas are given in the Appendix.

\section{Proof of the main theorem}
  \label{s:proof}

We recall that
$M=\Gamma\backslash\mathbb H^2$ is a convex co-compact hyperbolic quotient and that we 
fix a Schottky representation for $\Gamma$ as in~\S\ref{s:schottky}.
Finally we assume that the limit set $\Lambda_\Gamma$ satisfies fractal uncertainty principle
with exponent $\beta\geq 0$ in the sense of~\eqref{e:fup}. 

As seen in~\eqref{eq:delLs}, the main theorem
follows from showing that 
for
\begin{equation}
  \label{e:the-s}
s:={1\over 2}-\nu+ {i\over h},\quad
0\leq \nu\leq \beta-2\varepsilon
\end{equation}
we have 
\begin{equation}
  \label{e:the-u}
u\in\mathcal H(\mathbf D),\quad
\mathcal L_s u=u \ \ \Longrightarrow \ \ u \equiv 0 . 
\end{equation}
That follows from progressively obtaining more and more bounds on solutions
to the equation $ \mathcal L_s u = u $,
until we can use the fractal uncertainty principle~\eqref{e:fup} to show that $u\equiv 0$.

As shown in~\cite[\S5.1]{hgap}, we always have $\beta\leq {1\over 2}$. This implies
the inequality $\Re s\geq 0$ which we occasionally use below.

\subsection{A priori bounds}
  \label{s:a-priori}

We first use the equation in \eqref{e:the-u} to establish some a priori bounds on $u$.
Note that $u=\mathcal L_s u$ implies that $u$ extends holomorphically to a neighbourhood
of $\mathbf D$ and in particular is smooth up to the boundary of $\mathbf D$.

The prefactors $\gamma'_a(z)^s$ in~\eqref{e:transfer} are exponentially
large in $h$ when $z$ is not on the real line. To balance this growth we 
introduce the weight
$$
w_K:\mathbb C\to (0,\infty),\quad
w_K(z)=e^{ -{K |\Im z|}/{ h}} . 
$$
The constant $K$ is chosen so that a sufficiently high power of the transfer operator
$\mathcal L_s$ is bounded uniformly in $h$ on the weighted space induced by $w_K$:
\begin{lemm}
  \label{l:mapper}
There exist $K\geq 10$ and $n_0\in\mathbb N$ depending only on the Schottky data
so that, with $\mathbf D_2\subset\mathbf D^\circ$ defined in~\eqref{e:D-2},
\begin{equation}
  \label{e:mapper}
\sup_{\mathbf D}|w_K\mathcal L_s^{n_0}f|\leq C_\Gamma \sup_{\mathbf D_2}|w_K f|\quad\text{for all }
f\in \mathcal H(\mathbf D).
\end{equation}
\end{lemm}
\begin{proof}
Using Lemma~\ref{l:gamma-der} and the contraction property~\eqref{e:contracting},
we choose $n_0$ such that
$$
\sup_{D_b}|\gamma'_{\mathbf a}|\leq \tfrac{1}{ 2}\quad\text{for all}\quad
\mathbf a\in\mathcal W_{n_0},\quad
\mathbf a\to b\in\mathcal A.
$$
(See \eqref{eq:Wn} for definitions.)
In particular, since $\gamma_{\mathbf a}$ maps the real line to itself, we have
$$
|\Im\gamma_{\mathbf a}(z)|\leq \tfrac{1}{2}|\Im z|\quad\text{for all}\quad
\mathbf a\in \mathcal W_{n_0},\quad
\mathbf a\to b\in\mathcal A,\quad
z\in D_b.
$$
Take $z\in D_b$ for some $b\in \mathcal A$. Then
$$
\begin{aligned}
|w_K\mathcal L_s^{n_0}f(z)|&\leq \sum_{\mathbf a\in \mathcal W_{n_0}\atop\mathbf a\to b}
w_K(z)\big|\gamma_{\mathbf a}'(z)^s f(\gamma_{\mathbf a}(z))\big|
\leq\sum_{\mathbf a\in \mathcal W_{n_0}\atop\mathbf a\to b}
{w_K(z)|\gamma_{\mathbf a}'(z)^s|\over w_K(\gamma_{\mathbf a}(z))}
\sup_{\mathbf D_2}|w_Kf|
\\&
\leq C_\Gamma\sum_{\mathbf a\in \mathcal W_{n_0}\atop\mathbf a\to b}
e^{ - ( K |\Im z | + 2 \arg \gamma_{\mathbf a}' ( z ) )/2h } 
\sup_{\mathbf D_2}|w_Kf|
\end{aligned}
$$
where in the last inequality we use the bound $|\gamma'_{\mathbf a}(z)^s|\leq C_\Gamma e^{-\arg\gamma_{\mathbf a}'(z)/h} $.
It remains to choose $K$ large enough so that
$$
-\arg\gamma_{\mathbf a}'(z)\leq  \tfrac 1 2 {K|\Im z|}\quad\text{for all}\quad
\mathbf a\in \mathcal W_{n_0},\quad
\mathbf a\to b\in \mathcal A,\quad
z\in D_b
$$
which is possible since $\arg\gamma_{\mathbf a}'(z)=0$ when $z\in \mathbb R$.
\end{proof}
We next show that for $u$ satisfying~\eqref{e:the-u}, the norm
of $u$ on $\mathbf D$ is controlled by its norm on $\mathbf I=\mathbf D\cap \mathbb R$:
\begin{lemm}
\label{l:a-priori}
We have
\begin{equation}
  \label{e:a-priori}
\sup_{\mathbf D}|w_Ku|\leq C_\Gamma\sup_{\mathbf I}|u|.
\end{equation}
\end{lemm}
\begin{proof}
Let $n_0$ come from Lemma~\ref{l:mapper}. Since $u=\mathcal L_s^{n_0}u$, \eqref{e:mapper} shows
\begin{equation}
  \label{e:ap-1}
\sup_{\mathbf D}|w_Ku|\leq C_\Gamma\sup_{\mathbf D_2}|w_Ku|.
\end{equation}
Applying Lemma~\ref{l:interpolated-bound} to the functions
$\exp(\pm iK z/h)u(z)$ in $\mathbf D^\pm$, we get
for some $c\in (0,1]$ depending only on the Schottky data
\begin{equation}
  \label{e:ap-2}
\sup_{\mathbf D_2}|w_K u|\leq \big(\sup_{\mathbf I}|u|\big)^{c}
\big(\sup_{\mathbf D}|w_K u|\big)^{1-c}.
\end{equation}
Together~\eqref{e:ap-1} and~\eqref{e:ap-2} imply~\eqref{e:a-priori}.
\end{proof}

\subsection{Cutoffs and microlocalization}

Let $ Z(h^\rho)\subset\mathcal W^\circ$ be the partition defined in~\eqref{e:Z-tau}, 
where $ \rho \in ( 0 ,1 ) $ is fixed as in \eqref{e:fup}.
By Lemma~\ref{l:transfer-power}
\begin{equation}
  \label{e:u-transfer}
\mathcal L_s(h^\rho)u=u\quad\text{where}\quad
\mathcal L_s(h^\rho):=\mathcal L_{\overline{Z(h^\rho)},s}.
\end{equation}
That is,
\begin{equation}
  \label{e:transfer-reminder}
u(x)=
\mathcal L_s(h^\rho)u(x)=\sum_{\mathbf a\in Z(h^\rho)\atop \overline{\mathbf a}\rsa b}
\gamma'_{\overline{\mathbf a}'}(x)^s u(\gamma_{\overline{\mathbf a}'}(x)),\quad
x\in I_b,\
b\in\mathcal A.
\end{equation}
From~\cite[Lemma~2.10]{hyperfup} and Lemma~\ref{l:gamma-der}, we have
for all $\mathbf a=a_1\dots a_n\in Z(h^\rho)$
\begin{align}
  \label{e:Z-bound-1}
C_\Gamma^{-1}h^\rho&\leq |I_{\mathbf a}|\leq  h^\rho,
\\
  \label{e:Z-bound-2}
C_\Gamma^{-1}h^\rho&\leq |I_{\overline{\mathbf a}}|\leq C_\Gamma h^\rho,
\\
  \label{e:Z-bound-3}
C_\Gamma^{-1}h^\rho&\leq \gamma'_{\overline{\mathbf a}'}\leq C_\Gamma h^\rho\quad\text{on }I_{\overline{a_1}}.
\end{align}
Let $\mathbf I'$ be defined by~\eqref{e:D-def};
recall also the definition~\eqref{e:I-a-prime} of $I'_{\mathbf a}$.
Since $\Re s\geq 0$,
\eqref{e:transfer-reminder} gives
\begin{equation}
  \label{e:reduced-norm}
\sup_{\mathbf I}|u|\leq C_\Gamma\sup_{\mathbf I'}|u|\leq C_\Gamma \sum_{\mathbf a\in Z(h^\rho)}\sup_{I'_{\overline{\mathbf a}}}|u|
\end{equation}
where for the first inequality we used $u=\mathcal L_s u$.
That is, the behavior of $u$ on $\mathbf I$ (and thus by Lemma~\ref{l:a-priori},
on $\mathbf D$) is controlled by its behavior on the intervals $I'_{\overline{\mathbf a}}$
for $\mathbf a\in Z(h^\rho)$.

We now define pieces of $u$ localized to each interval $I_{\overline{\mathbf a}}$.
To that aim, for each $\mathbf a\in Z(h^\rho)$
we choose a cutoff function
\begin{equation}
  \label{e:chi-a}
\chi_{\mathbf a}\in \CIc(I_{\overline{\mathbf a}}^\circ;[0,1]),\quad
\supp(1-\chi_{\mathbf a})\cap I'_{\overline{\mathbf a}}=\emptyset,\quad
\sup|\partial^j_x\chi_{\mathbf a}|\leq C_{\Gamma,j}h^{-\rho j}.
\end{equation}
This is possible since $d_{\mathbb R}(I'_{\overline{\mathbf a}},\mathbb R\setminus I_{\overline{\mathbf a}})\geq C_\Gamma^{-1}h^\rho$
by~\eqref{e:Z-bound-3}. We then define
\begin{equation}
  \label{e:u-a}
u_{\mathbf a}:=\langle x\rangle^{2s}\chi_{\mathbf a}u\in \CIc(I_{\overline{\mathbf a}}^\circ),\quad
\mathbf a\in Z(h^\rho).
\end{equation}
Equation~\eqref{e:transfer-reminder} implies that, with $T_{\gamma,s}$ given by~\eqref{e:T-gamma},
\begin{equation}
  \label{e:transfer2}
\langle x\rangle^{2s}u(x)=\sum_{\mathbf a\in Z(h^\rho)\atop \overline{\mathbf a}\rsa b}
T_{\gamma_{\overline{\mathbf a}'},s}u_{\mathbf a}(x),\quad
x\in I'_b,\
b\in\mathcal A.
\end{equation}
The next lemma shows that $u_{\mathbf a}$'s are semiclassically localized
to a bounded set in $ \xi $.
This is where we use that $\rho<1$: if we instead put $\rho=1$ then
multiplication by $\chi_{\mathbf a}$ would inevitably blur the support of the semiclassical Fourier transform.
\begin{lemm}
  \label{l:u-fourier}
Let $K\geq 10$ be chosen in Lemma~\ref{l:mapper}.
Then for each $\mathbf a\in Z(h^\rho)$, $ u_{\mathbf a}$ 
is semiclassically localized to frequencies $  |\xi| \leq \frac{3K}2 $. More precisely, 
for all $N $,
\begin{equation}
\label{e:u-fourier}
| \mathcal F_h u_{\mathbf a} ( \xi ) | \leq C_{\Gamma, N } h^N |\xi|^{-N} \sup_{ 
\mathbf I} | u|  \quad\text{when }
\ |\xi|\geq \tfrac{3K}2 .
\end{equation}
\end{lemm}
\Remark A finer compact microlocalization statement can be given using the FBI transform~-- see~\cite[Proposition 2.2]{JinCauchy}.
\begin{proof}
For each $\mathbf a\in Z(h^\rho)$,
let $ \widetilde \chi_{\mathbf a} \in \CIc ( D_{\overline{\mathbf a}}^\circ) $ be an almost analytic extension of $\chi_{\mathbf a}$, more precisely for each $N$
$$
\begin{aligned}
\widetilde\chi_{\mathbf a}|_{\mathbb R}&=\chi_{\mathbf a},\ \ \ \ 
h^\rho|\bar\partial_z \widetilde\chi_{\mathbf a}(z)| \leq C_{\Gamma,N}\big(h^{-\rho}|\Im z|\big)^N.
\end{aligned} 
$$
See for instance~\cite[Theorem 3.6]{ev-zw} for a construction of such extension; here we map $D_{\overline{\mathbf a}}$
to the unit disk and use the derivative bounds~\eqref{e:chi-a}.
Let $\langle z\rangle^{2s}:=(1+z^2)^s$ be the holomorphic extension of $\langle x\rangle^{2s}$
to $D_{\overline{\mathbf a}}$. We note that
\[ |\langle z\rangle^{2s}|\leq C_\Gamma\exp(2|\Im z|/h), \ \ z\in D_{\overline{\mathbf a}} . \]
Since $u$ is holomorphic in~$D_{\overline{\mathbf a}}$, we have
\[
\begin{split}
\mathcal F_h u_{\mathbf a}({\xi})  & = 
\int_{\mathbb R} u ( x ) \langle x\rangle^{2s}e^{ - \frac i h \xi x }\widetilde\chi_{\mathbf a} ( x )  \, dx 
\\ & 
=-{(\sgn\xi)} \int_{D_{\overline{\mathbf a}}\cap \{(\sgn\xi)\Im z\leq 0\}} 
 u ( z ) \langle z\rangle^{2s}e^{ - \frac i h \xi z }\bar \partial_z  \widetilde\chi_{\mathbf a}  ( z ) \,d\bar z\wedge dz . 
 \end{split} \]
Using Lemma~\ref{l:a-priori} and the fact that $|\xi|\geq 3K/2$ we estimate for all $N$,
$$
\begin{aligned}
|\mathcal F_h u_{\mathbf a}({\xi})| &\leq C_{\Gamma,N}\sup_{\mathbf I}|u|\int_0^{\infty}
(h^{-\rho}y)^N
e^{(K-|\xi|+2)y/ h}
\,dy
\\&
\leq C_{\Gamma,N}\sup_{\mathbf I}|u|\int_0^\infty (h^{-\rho}y)^N
e^{-|\xi|y/ 5h} 
\,dy
\leq C_{\Gamma,N}h^{1+N(1-\rho)}|\xi|^{-1-N}\sup_{\mathbf I}|u|.
\end{aligned}
$$
Since $\rho<1$ and $N$ can be chosen arbitrary, \eqref{e:u-fourier} follows.
\end{proof}
Lemma~\ref{l:u-fourier} implies that the sup-norm of $u_{\mathbf a}$
can be estimated by its $L^2$ norm:
\begin{lemm}
  \label{l:l2-reduced}
We have for all $N$ and $\mathbf a\in Z(h^\rho)$,
\begin{equation}
  \label{e:l2-reduced}
\sup|u_{\mathbf a}|\leq C_\Gamma h^{-1/2}\|u_{\mathbf a}\|
+C_{\Gamma,N}h^N\sup_{\mathbf I}|u|.
\end{equation}
\end{lemm}
\begin{proof}
We use the Fourier inversion formula:
$
u_{\mathbf a}(x)=(2\pi h)^{-1}\int_{\mathbb R}e^{{i}x\xi/h}\, \mathcal F_h u_{\mathbf a} ( \xi ) \,d\xi.
$
We split this integral into two parts. The first one is the integral over $\{|\xi|\leq \frac{3K}2\}$, which
is bounded by $C_\Gamma h^{-1} \|\mathcal F_h u_{\mathbf a}\| = C_\Gamma 
h^{-1/2}\|u_{\mathbf a}\|$.
The second one, the integral over $\{|\xi|\geq \frac{3K} 2\}$, is bounded by
$C_{\Gamma,N}h^N\sup_{\mathbf I}|u|$ by Lemma~\ref{l:u-fourier}.
\end{proof}
Since the intervals $I_{\mathbf a}$, $\mathbf a\in Z(h^\rho)$, do not intersect
and are contained in $\mathbf I$, we have $\#(Z(h^\rho))\leq C_\Gamma h^{-1}$.
Combining this with~\eqref{e:reduced-norm} and~\eqref{e:l2-reduced}, we get
the following 
bound:
\begin{equation}
  \label{e:sup-l2}
\sup_{\mathbf I}|u|
\leq C_\Gamma \sum_{\mathbf a\in Z(h^\rho)}\sup|u_{\mathbf a}|
\leq C_\Gamma h^{-1}\bigg(\sum_{\mathbf a\in Z(h^\rho)}\|u_{\mathbf a}\|^2\bigg)^{1/2}.
\end{equation}

\subsection{End of the proof}
  \label{s:proof-end}

We use \eqref{e:transfer2} together with equivariance of the operator $\mathcal B(s)$
(Lemma~\ref{l:B-equiv}) to obtain a formula for $u$ in terms of $\mathcal B(s)$,
see~\eqref{e:u-formula} below. Together with the fractal uncertainty bound~\eqref{e:fup}
this will give $u\equiv 0$.

In order to take advantage of the equivariance of $\mathcal B(s)$, we approximate  $u_{\mathbf a}$'s
appearing in \eqref{e:transfer2} by functions in the range of
$\mathcal B(s)$.
For that, let $K$ be chosen from Lemma~\ref{l:mapper}. 
Define the partition $Z({1\over 10K})$ by~\eqref{e:Z-tau}.
Since $h$ is small and by~\eqref{e:Z-bound-2},
\begin{equation}
\label{eq:defat}
\forall \, \mathbf a\in Z(h^\rho) \quad \exists ! \ \widetilde{\mathbf a}\in Z\big(\tfrac{1}{10K}\big):\quad
\widetilde{\mathbf a}\prec\overline{\mathbf a},\quad
\widetilde{\mathbf a}\neq\overline{\mathbf a}.
\end{equation}
We stress that $10K|I_{\widetilde{\mathbf a}}|\leq 1$ and that
 $\widetilde{\mathbf a}$ lies in~$Z({1\over 10K})$ which is a finite
$h$-independent set.
Choose $h$-independent cutoff functions (see \eqref{e:I-a-prime}) 
$$
\begin{aligned}
\chi_{\widetilde{\mathbf a}}\in \CIc(I_{\widetilde{\mathbf a}}^\circ),&\quad
\supp(1-\chi_{\widetilde{\mathbf a}})\cap I'_{\widetilde{\mathbf a}}=\emptyset;\\
\chi_K\in \CIc((-2K,2K)),&\quad
\supp(1-\chi_K)\cap \big[-\tfrac{3K}{2},\tfrac{3K}{2}\big]=\emptyset
\end{aligned}
$$
and consider semiclassical pseudodifferential operators
$$
A_{\widetilde{\mathbf a}}:=\chi_{\widetilde{\mathbf a}}(x)\chi_K(hD_x)
=\Op_h\big(\chi_{\widetilde{\mathbf a}}(x)\chi_K(\xi)\big).
$$
By Lemma~\ref{l:u-fourier} and since $\supp u_{\mathbf a}\subset I_{\overline{\mathbf a}}\subset I'_{\widetilde{\mathbf a}}$,
we have for all $\mathbf a\in Z(h^\rho)$ and all~$N$
\begin{equation}
  \label{e:cooking-1}
\|u_{\mathbf a}-A_{\widetilde{\mathbf a}}u_{\mathbf a}\|\leq C_{\Gamma,N}h^N\sup_{\mathbf I}|u|.
\end{equation}
We now apply Lemma~\ref{l:B-invert} with $I:=I_{\widetilde{\mathbf a}}$ and $A:=A_{\widetilde{\mathbf a}}$
to write
\begin{equation}
  \label{e:cooking-2}
\begin{gathered}
A_{\widetilde{\mathbf a}}=\psi_{\widetilde{\mathbf a}}\mathcal B(s)\omega_{\widetilde{\mathbf a}}Q_{\widetilde{\mathbf a}}+\mathcal O(h^\infty)_{L^2(\mathbb R)\to L^2(\mathbb R)} 
\end{gathered}
\end{equation}
where $Q_{\widetilde{\mathbf a}}:L^2(\mathbb R)\to L^2(\dot{\mathbb R})$ is bounded
uniformly in $h$ and 
\begin{equation}
  \label{e:cooking-cutoffs}
\psi_{\widetilde{\mathbf a}}\in \CIc(I_{\widetilde{\mathbf a}}^\circ),\quad
\supp\chi_{\widetilde{\mathbf a}}\subset \{\psi_{\widetilde{\mathbf a}}=1\},\quad
\omega_{\widetilde{\mathbf a}}\in C^\infty(\dot{\mathbb R}\setminus I_{\widetilde{\mathbf a}}).
\end{equation}
Define (see Figure~\ref{f:end-cutoffs})
\begin{equation}
  \label{e:v-a-def}
v_{\mathbf a}:=\omega_{\widetilde{\mathbf a}}Q_{\widetilde{\mathbf a}}u_{\mathbf a}\in L^2(\dot{\mathbb R}),\quad
\|v_{\mathbf a}\|_{L^2(\dot{\mathbb R})}\leq C_\Gamma \|u_{\mathbf a}\|_{L^2(\mathbb R)}.
\end{equation}
Then~\eqref{e:cooking-1} and~\eqref{e:cooking-2} imply that for all $N$,
\begin{equation}
  \label{e:cooking-3}
\|u_{\mathbf a}-\psi_{\widetilde{\mathbf a}}\mathcal B(s)v_{\mathbf a}\|
\leq C_{\Gamma,N}h^N\sup_{\mathbf I}|u|.
\end{equation}
\begin{figure}
\includegraphics{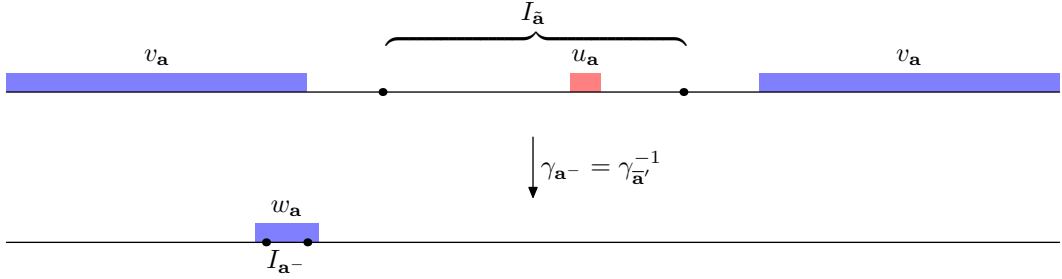}
\caption{The supports of $u_{\mathbf a},v_{\mathbf a},w_{\mathbf a}$ (shaded)
and the intervals $I_{\widetilde{\mathbf a}},I_{\mathbf a^-}$ (marked by endpoints).
The top picture is mapped to the bottom one by $\gamma_{\mathbf a^-}$.
The supports of $u_{\mathbf a},w_{\mathbf a}$ have size $\sim h^\rho$.}
\label{f:end-cutoffs}
\end{figure}

To approximate $ u $ by an element of the image of 
$ \mathcal B ( s ) $ we  use the equation~\eqref{e:transfer2} and 
the equivariance property~\eqref{e:B-equivariance}. Let $\mathbf a\in Z(h^\rho)$, $b\in\mathcal A$,
$\overline{\mathbf a}\rsa b$. By~\eqref{e:cooking-cutoffs}, we have
$\psi_{\widetilde{\mathbf a}}=1$ on $\gamma_{\overline{\mathbf a}'}(I_b)=I_{\overline{\mathbf a}}\subset I'_{\widetilde{\mathbf a}}$.
Thus by~\eqref{e:B-equivariance}
\begin{equation}
  \label{e:cooking-4}
T_{\gamma_{\overline {\mathbf a}'},s}\psi_{\widetilde{\mathbf a}}\mathcal B(s)v_{\mathbf a}
=\mathcal B(s)T_{\gamma_{\overline{\mathbf a}'},1-s}v_{\mathbf a}\quad\text{on }I_b.
\end{equation}
Substituting~\eqref{e:cooking-3} into~\eqref{e:transfer2} and using~\eqref{e:cooking-4},
we obtain the following approximate formula for $u$, true for each $N$
and $b\in\mathcal A$:
\begin{equation}
  \label{e:u-formula}
\Big\| \langle x\rangle^{2s}u - \mathcal B(s)\sum_{\mathbf a\in Z(h^\rho)\atop \overline{\mathbf a}\rsa b}
w_{\mathbf a}\Big\|_{L^2(I'_b)}\leq C_{\Gamma,N}h^N \sup_{\mathbf I}|u|
\end{equation}
where
\begin{equation}
  \label{e:w-a-def}
w_{\mathbf a}:=T_{\gamma_{\overline{\mathbf a}'},1-s}v_{\mathbf a}\in L^2(\dot{\mathbb R}).
\end{equation}
We now establish a few properties of $w_{\mathbf a}$, starting with its support:
\begin{lemm}
  \label{l:w-supp}
Let $\mathbf a=a_1\dots a_n\in Z(h^\rho)$ and put $\mathbf a^-:=a_2\dots a_n\in\mathcal W^\circ$.
Then
\begin{equation}
  \label{e:w-supp}
\supp w_{\mathbf a}\subset I_{\mathbf a^-}(C_\Gamma h^\rho)
\end{equation}
where $I(\tau)=I+[-\tau,\tau]$ denotes the $\tau$-neighbourhood of $I$.
\end{lemm}
\begin{proof}
From the definition of $w_{\mathbf a}$ and~\eqref{e:cooking-cutoffs} we have
(see Figure~\ref{f:end-cutoffs})
$$
\supp w_{\mathbf a}\subset \gamma_{\overline{\mathbf a}'}^{-1}(\supp v_{\mathbf a})
\subset \gamma_{\mathbf a^-}(\dot{\mathbb R}\setminus I_{\widetilde{\mathbf a}}).
$$
Recall that $\widetilde{\mathbf a}\prec\overline{\mathbf a}$, thus
$\widetilde{\mathbf a}=\overline{a_n}\,\overline{a_{n-1}}\dots \overline{a_{n-m}}$
for some $m\leq C_\Gamma$, in particular $m\ll n$. Then by~\eqref{e:gamma-a-property}
we have $\gamma_{a_{n-m}}(\dot{\mathbb R}\setminus I_{\overline{a_{n-m}}})=I_{a_{n-m}}^\circ$ and thus
$$
\gamma_{\mathbf a^-}(\dot{\mathbb R}\setminus I_{\widetilde{\mathbf a}})
=\gamma_{a_2}\cdots \gamma_{a_{n-m}}(\dot{\mathbb R}\setminus I_{\overline{a_{n-m}}})
=I_{a_2\dots a_{n-m}}^\circ.
$$
We have $I_{a_2\dots a_{n-m}}\supset I_{\mathbf a^-}$ and by~\cite[Lemma~2.7]{hyperfup} we have
$|I_{a_2\dots a_{n-m}}|\leq C_\Gamma |I_{\mathbf a^-}|\leq C_\Gamma h^\rho$.
Therefore $I_{a_2\dots a_{n-m}}\subset I_{\mathbf a^-}(C_\Gamma h^\rho)$, finishing the proof.
\end{proof}
We next estimate the norm of $w_{\mathbf a}$. This is where the value of the parameter
$\nu$ from~\eqref{e:the-s} enters the argument.
\begin{lemm}
  \label{l:w-norm}
We have
\begin{equation}
  \label{e:w-norm}
\|w_{\mathbf a}\|
\leq C_\Gamma h^{-\rho\nu} \|u_{\mathbf a}\|.
\end{equation}
\end{lemm}
\begin{proof}
The definition~\eqref{e:w-a-def} of $w_{\mathbf a}$ gives (recalling~\eqref{e:T-gamma})
$
w_{\mathbf a}(x)=|\gamma'_{\overline{\mathbf a}'}(x)|_{\mathbb S}^{1-s}v_{\mathbf a}(\gamma_{\overline{\mathbf a}'}(x))$ for 
$ x\in\dot{\mathbb R}.
$
Using the notation of the proof of Lemma~\ref{l:w-supp}, we can restrict
to $x\in I_{a_2\dots a_{n-m}}$.
Since $m\leq C_\Gamma$ we have
$$
C_\Gamma^{-1}|\gamma'_{\overline{a_{n-m-1}}
\dots \overline{a_2}}(x)|_{\mathbb S}\leq 
|\gamma'_{\overline{\mathbf a}'}(x)|_{\mathbb S}\leq C_\Gamma|\gamma'_{\overline{a_{n-m-1}}
\dots \overline{a_2}}(x)|_{\mathbb S}.
$$
We have for $x\in I_{a_2\dots a_{n-m}}$
$$
|\gamma'_{\overline{a_{n-m-1}}\dots \overline{a_2}}(x)|_{\mathbb S}=
|\gamma'_{a_2\dots a_{n-m-1}}(x')|_{\mathbb S}^{-1},\quad
x':=\gamma_{\overline{a_{n-m-1}}\dots \overline{a_2}}(x)\in I_{a_{n-m}}.
$$
By Lemma~\ref{l:gamma-der} and~\cite[Lemma~2.7]{hyperfup} we have
$$
C_\Gamma^{-1}h^\rho\leq |\gamma'_{a_2\dots a_{n-m-1}}(x')|_{\mathbb S}\leq C_\Gamma h^\rho,
$$
therefore
$$
C_\Gamma^{-1}h^{-\rho}\leq |\gamma'_{\overline{\mathbf a}'}(x)|_{\mathbb S}\leq C_\Gamma h^{-\rho},\quad
x\in I_{a_2\dots a_{n-m}},
$$
which using the change of variables formula and \eqref{e:the-s} implies
$$
\|w_{\mathbf a}\|_{L^2(\mathbb R)}\leq C_\Gamma h^{-\rho\nu}\|v_{\mathbf a}\|_{L^2(\dot{\mathbb R})}.
$$
Together with~\eqref{e:v-a-def} this gives~\eqref{e:w-norm}.
\end{proof}
We next rewrite the formula~\eqref{e:u-formula} in terms of the operator $\mathcal B_\chi(h)$
defined in~\eqref{e:B-chi}. Let $\chi_0\in C^\infty(\dot{\mathbb R}\times\dot{\mathbb R})$ be such that 
$$ \supp\chi_0\cap \{(x,x)\mid x\in\dot{\mathbb R}\}=\emptyset $$ 
and 
\begin{equation*}
a,b\in\mathcal A,\
a\neq b \ \ \Longrightarrow \ \ 
\supp(1-\chi_0)\cap (I_a\times I_b) =\emptyset .
\end{equation*}
Put
\begin{equation}
  \label{e:the-chi}
\chi(x,x'):=|x-x'|_{\mathbb S}^{2\nu-1} \chi_0(x,x').
\end{equation}
For each $b\in\mathcal A$ define the following compactly supported function
$$
w^{(b)}:=\sum_{\mathbf a\in Z(h^\rho)\atop \overline{\mathbf a}\rsa b}w_{\mathbf a}\in L^2(\mathbb R).
$$
Then by~\eqref{e:w-supp} we have
(identifying $\dot{\mathbb R}$ with $\mathbb S^1$ using~\eqref{e:line-circle})
$$
\mathcal B_\chi(h)w^{(b)}=\mathcal B(s)\sum_{\mathbf a\in Z(h^\rho)\atop\overline{\mathbf a}\rsa b}w_{\mathbf a}\quad\text{on }I_b.
$$
Thus~\eqref{e:u-formula} implies
\begin{equation}
  \label{e:u-formula-2}
\big\|\langle x\rangle^{2s}u-\mathcal B_\chi(h)w^{(b)}\big\|_{L^2(I'_b)}\leq C_{\Gamma,N}h^N\sup_{\mathbf I}|u|.
\end{equation}
The supports of $w_{\mathbf a}$, as well as the supports of $u_{\mathbf a}$,
lie in a $C_\Gamma h^\rho$ neighbourhood of the limit set  and have bounded overlaps. Analysing this closely will give us
\begin{lemm}
  \label{l:terminator}
Denote by $\Lambda_\Gamma(\tau)=\Lambda_\Gamma+[-\tau,\tau]$ the $\tau$-neighbourhood
of the limit set $\Lambda_\Gamma\subset\mathbb R$. Then
\begin{align}
  \label{e:terminator-1}
\bigcup_{b\in\mathcal A}\supp w^{(b)}&\subset \Lambda_\Gamma(C_\Gamma h^\rho),\\
  \label{e:terminator-2}
\max_{b\in\mathcal A}\|w^{(b)}\|^2 &\leq C_\Gamma\sum_{\mathbf a\in Z(h^\rho)}\|w_{\mathbf a}\|^2,\\
  \label{e:terminator-3}
\sum_{\mathbf a\in Z(h^\rho)}\|u_{\mathbf a}\|^2
&\leq C_\Gamma \|\indic_{\Lambda_\Gamma(C_\Gamma h^\rho)}\langle x\rangle^{2s}u\|^2.
\end{align}
\end{lemm}
\begin{proof}
By Lemma~\ref{l:w-supp}, for each $\mathbf a\in Z(h^\rho)$ we have
$\supp w_{\mathbf a}\subset I_{\mathbf a^-}(C_\Gamma h^\rho)$.
By~\cite[Lemma~2.7]{hyperfup} the interval $I_{\mathbf a^-}$ has length
$\leq C_\Gamma h^\rho$ and it intersects the limit set,
therefore $\supp w_{\mathbf a}\subset \Lambda_\Gamma(C_\Gamma h^\rho)$.
This proves~\eqref{e:terminator-1}.

Next, we have the following multiplicity estimates:
\begin{align}
  \label{e:mul-1}
\sup_{x\in\mathbb R}\#\{\mathbf a\in Z(h^\rho)\mid x\in \supp w_{\mathbf a}\}&\leq C_\Gamma,
\\  \label{e:mul-2}
\sup_{x\in\mathbb R}\#\{\mathbf a\in Z(h^\rho)\mid x\in \supp \chi_{\mathbf a}\}&\leq C_\Gamma.
\end{align}
Both of these follow from Lemma~\ref{l:multiplicity}. Indeed, from the proof of Lemma~\ref{l:w-supp}
we see that for $\mathbf a=a_1\dots a_n\in Z(h^\rho)$, we have $\supp w_{\mathbf a}\subset I_{a_2\dots a_{n-m}}$
where $m$ is fixed large enough depending only on the Schottky data.
By~\cite[Lemma~2.7]{hyperfup} we have
$C_\Gamma^{-1}h^\rho\leq |I_{a_2\dots a_{n-m}}|\leq C_\Gamma h^\rho$.
Therefore, the number of intervals $I_{a_2\dots a_{n-m}}$ containing a given point $x$
is bounded, which gives~\eqref{e:mul-1}.
To prove~\eqref{e:mul-2}, we use that $\supp \chi_{\mathbf a}\subset I_{\overline{\mathbf a}}$
by~\eqref{e:chi-a}
and $C_\Gamma^{-1}h^\rho\leq |I_{\overline{\mathbf a}}|\leq C_\Gamma h^\rho$ by~\eqref{e:Z-bound-2}.

Now, \eqref{e:mul-1} immediately gives~\eqref{e:terminator-2}:
$$
\max_{b\in\mathcal A}\|w^{(b)}\|_{L^2(\mathbb R)}^2\leq
\int_{\mathbb R}\Big(\sum_{\mathbf a\in Z(h^\rho)}|w_{\mathbf a}(x)|\Big)^2\,dx
\leq C_\Gamma\int_{\mathbb R}\sum_{\mathbf a\in Z(h^\rho)}|w_{\mathbf a}(x)|^2\,dx.
$$
To show~\eqref{e:terminator-3}, we first note that for all $\mathbf a\in Z(h^\rho)$
we have $\supp \chi_{\mathbf a}\subset I_{\overline{\mathbf a}}\subset\Lambda_\Gamma(C_\Gamma h^\rho)$,
since $I_{\overline{\mathbf a}}$ is an interval of size $\leq C_\Gamma h^\rho$ 
intersecting the limit set. We now recall~\eqref{e:u-a}:
$$
\begin{aligned}
\sum_{\mathbf a\in Z(h^\rho)}\|u_{\mathbf a}\|_{L^2(\mathbb R)}^2
&=\int_{\Lambda_\Gamma(C_\Gamma h^\rho)}\sum_{\mathbf a\in Z(h^\rho)}|\chi_{\mathbf a}(x)|^2
\cdot |\langle x\rangle^{2s}u(x)|^2\,dx
\\&\leq
C_\Gamma\int_{\Lambda_\Gamma(C_\Gamma h^\rho)}|\langle x\rangle^{2s}u(x)|^2\,dx
\end{aligned}
$$
where in the last inequality we used~\eqref{e:mul-2}.
\end{proof}
We are now ready to prove \eqref{e:the-u} and hence finish the proof of the main theorem. From~\eqref{e:u-formula-2}
and the fact that $\Lambda_\Gamma(C_\Gamma h^\rho)\subset \bigcup_{b\in\mathcal A}I'_b$ we obtain
\begin{equation}
\label{eq:f1}  \|\indic_{\Lambda_\Gamma(C_\Gamma h^\rho)}\langle x\rangle^{2s}u\|^2
\\
\leq C_\Gamma\max_{b\in\mathcal A}\|\indic_{\Lambda_\Gamma(C_\Gamma h^\rho)}\mathcal B_\chi(h)w^{(b)}\|^2
+\mathcal O(h^\infty)\sup_{\mathbf I}|u|^2 .\end{equation}
The fractal uncertainty bound~\eqref{e:fup}
and the support property~\eqref{e:terminator-1} show that
\begin{equation}
\label{eq:f2}
\max_{b\in\mathcal A}\|\indic_{\Lambda_\Gamma(C_\Gamma h^\rho)}\mathcal B_\chi(h)w^{(b)}\|^2 
\leq 
C_\Gamma h^{2(\beta-\varepsilon)}\max_{b\in\mathcal A}\|w^{(b)}\|^2 . 
\end{equation}
The estimate~\eqref{e:terminator-2} and Lemma ~\ref{l:w-norm} give
\begin{equation}
\label{eq:f3}
h^{2(\beta-\varepsilon)}\max_{b\in\mathcal A}\|w^{(b)}\|^2  
\leq
C_\Gamma h^{2(\beta-\varepsilon-\rho\nu)}\sum_{\mathbf a\in Z(h^\rho)}\|u_{\mathbf a}\|^2. 
\end{equation}
Due to~\eqref{e:terminator-3}, the left hand side of \eqref{eq:f1} bounds the sum on the right hand side of \eqref{eq:f3}. Putting \eqref{eq:f1},
\eqref{eq:f2}, and~\eqref{eq:f3} together and using \eqref{e:sup-l2} to remove
$ \mathcal O ( h^\infty ) \sup_{\mathbf I} | u|^2 $, we obtain for $h$ small enough
\begin{equation}
  \label{e:ultimate}
\sum_{\mathbf a\in Z(h^\rho)}\|u_{\mathbf a}\|^2
\leq C_\Gamma h^{2(\beta-\varepsilon-\rho\nu)}\sum_{\mathbf a\in Z(h^\rho)}\|u_{\mathbf a}\|^2 .
\end{equation}
From~\eqref{e:the-s} we have $2(\beta-\varepsilon-\rho\nu)>0$, thus \eqref{e:ultimate}
implies that $\sum_{\mathbf a\in Z(h^\rho)}\|u_{\mathbf a}\|^2=0$ if $ 
h $ is small enough.
Recalling \eqref{e:u-a}, analyticity of $ u $ shows that 
$ u \equiv 0 $.

\section*{Appendix}

\refstepcounter{section}
\renewcommand{\thesection}{A}

We give proofs of the two lemmas about the operator $ \mathcal B ( s) $
introduced in \S \ref{s:integral-operator}.

\begin{proof}[Proof of Lemma \ref{l:B-s-bdd}]
For simplicity we assume that $\supp\chi_1,\supp\chi_2\subset \mathbb R$;
same proof applies to the general case, identifying $\dot{\mathbb R}$ with the circle
by~\eqref{e:line-circle}. The operator $(\chi_1\mathcal B(s)\chi_2)^*\chi_1\mathcal B(s)\chi_2$
has integral kernel
\begin{equation}
  \label{e:integral-kernel}
\begin{aligned}
\mathcal K(x,x'')=(2\pi h)^{-1}\int_{\dot{\mathbb R}}&
e^{{i\over h}(\Phi(x',x'')-\Phi(x',x))}|x-x'|^{2\nu-1}_{\mathbb S}|x'-x''|^{2\nu-1}_{\mathbb S}
\\&
|\chi_1(x')|^2\overline{\chi_2(x)}\chi_2(x'')\,dP(x').
\end{aligned}
\end{equation}
We have
$
\partial_{x'}\Phi(x',x)={2( x-x')^{-1}}+{2x' \langle x'\rangle^{-2}},
$
thus
$$
\big|\partial_{x'}\big(\Phi(x',x'')-\Phi(x',x)\big)\big|\geq {|x-x''|/ C}\quad\text{for }
x,x''\in\supp\chi_2,\
x'\in\supp\chi_1.
$$
Integrating by parts $N$ times in~\eqref{e:integral-kernel} we see that
$
|\mathcal K(x,x'')|\leq C_N h^{N-1}|x-x''|^{-N}.
$
Therefore
$
\sup_x\int_{\dot{\mathbb R}} |\mathcal K(x,x'')|\,dP(x'')\leq C$.
(Here we use the case $N=0$ for $|x-x''|\leq h$ and
the case $N=2$ for $|x-x''|\geq h$.)
By Schur's Lemma (see for instance \cite[Proof of Theorem 4.21]{ev-zw}) we see that
the operator
$(\chi_1\mathcal B(s)\chi_2)^*\chi_1\mathcal B(s)\chi_2$,
and thus $\chi_1\mathcal B(s)\chi_2$, is bounded on $L^2(\dot{\mathbb R})$ uniformly
in $h$.
\end{proof}

The following  technical lemma is useful to establish invertibility of $\mathcal B(s)$
in Lemma~\ref{l:B-invert}. To state it we use the standard (left) quantization
procedure \eqref{eq:qu}.
\begin{lemm}
  \label{l:bi-helper}
Assume that $\chi_1,\chi_3\in \CIc(\mathbb R)$ and $\chi_2\in C^\infty(\dot{\mathbb R})$
satisfy 
$$ \supp\chi_1\cap\supp\chi_2=\supp\chi_2\cap\supp\chi_3=\emptyset . $$
Let $q(x,\xi)\in \CIc(\mathbb R^2)$.
Define the operator
$$
B=\chi_1\mathcal B(s)\chi_2\mathcal B(1-s)\chi_3\Op_h(q):L^2(\mathbb R)\to L^2(\mathbb R).
$$
Then $B=\Op_h(b)$ where $b(x,\xi;h) \in \mathscr S ( \mathbb R^2 )$
(uniformly in~$h$) admits an asymptotic expansion:
\begin{equation}
  \label{e:bi-helper-exp}
\begin{gathered}
b(x,\xi;h)\sim \chi_1(x)\sum_{j=0}^\infty h^j L_j\big(\chi_2(x')\chi_3(x)q(x,\xi)\big)\big|_{x'=x'(x,\xi)}, \\
x'(x,\xi):=x+{2\langle x\rangle^2\over \langle x\rangle^2\xi-2x}\in\dot{\mathbb R},
\end{gathered}
\end{equation}
where $L_j$ is an order $2j$ differential operator on $\mathbb R_{x}\times \dot{\mathbb R}_{x'}$ with coefficients
depending on~$x,\xi,\nu$, and
$L_0=1/2$.
\end{lemm}
\begin{proof}
For simplicity we assume that $\supp\chi_2\subset \mathbb R$. The general case is handled
similarly, identifying $\dot{\mathbb R}$ with the circle.

For $\xi\in\mathbb R$, define the function $e_\xi(x)=\exp(ix\xi/h)$. By oscillatory testing~\cite[Theorem~4.19]{ev-zw},
we have $B=\Op_h(b)$ where $b$ is defined by the formula
$$
(Be_\xi)(x)=b(x,\xi;h)e_\xi(x).
$$
It remains to show that $b \in \mathscr S $ and had the expansion~\eqref{e:bi-helper-exp}. We compute
$$
b(x,\xi;h)=(2\pi h)^{-1}\int_{\mathbb R^2}e^{{i\over h}\Psi(x',x'';x,\xi)}
p(x',x'';x,\xi)\,dx'dx''
$$
where
$$
\begin{aligned}
\Psi&=\Phi(x,x')-\Phi(x',x'')+(x''-x)\xi,\\
p&=4\chi_1(x)\chi_2(x')\chi_3(x'')q(x'',\xi)|x-x'|_{\mathbb S}^{2\nu-1}
|x'-x''|_{\mathbb S}^{-2\nu-1}\langle x'\rangle^{-2}\langle x''\rangle^{-2}.
\end{aligned}
$$
We have
$$
\partial_{x'}\Psi
={2(x''-x)\over (x-x')(x''-x')},  \quad
\partial_{x''}\Psi
=\xi-{2x''\over\langle x''\rangle^2}+{2\over x''-x'}.
$$
It follows that $\Psi$ is a Morse function on $\{x'\neq x,\ x'\neq x''\}$ with the unique critical point given
by $x''=x$, $x'=x'(x,\xi)$ where $x'(x,\xi)$ is defined in~\eqref{e:bi-helper-exp}.

We have $\supp b\subset \{x\in\supp \chi_1\}$. Next,
for $(x',x'')\in\supp p$ and large $|\xi|$, we have
$|\partial_{x''}\Psi|\geq {1\over 2}|\xi|$. Therefore, repeated integration
by parts in $x''$ shows that $b=\mathcal O(h^\infty)_{\mathscr S(\mathbb R^2)}$ when $|x|+|\xi|$ is large.
The expansion~\eqref{e:bi-helper-exp} follows from the method of stationary phase.
To compute $L_0$ we use that the Hessian of $\Psi$ at the critical point
has signature~0 and determinant $4(x-x'(x,\xi))^{-4}$.
\end{proof}

\begin{proof}[Proof of Lemma \ref{l:B-invert}]
Choose $\psi_I,\omega_I$ satisfying the conditions in the statement of the lemma. 
It follows from the definition of $x'(x,\xi)$ in~\eqref{e:bi-helper-exp} that
$
|x-x'(x,\xi)|\geq {2/( 1+|\xi|)}.
$
Since $10K|I|\leq 1$ we have
\begin{equation}
  \label{e:bi-1}
x'(x,\xi)\notin\supp(1-\omega_I)\quad\text{for all }
x\in I,\quad
|\xi|\leq 2K.
\end{equation}
We now put
$$
Q_I (s):=\omega_I \mathcal B(1-s)\psi_I\Op_h(q):L^2(\mathbb R)\to L^2(\dot{\mathbb R})
$$
where $q(x,\xi;h)\in \CIc(\mathbb R^{2})$, to be chosen later, satisfies the support condition
$$
\supp q\subset \supp a\subset \{\psi_I(x)=1,\ |\xi|\leq 2K\}
$$
and has an asymptotic expansion
$$
q(x,\xi;h)\sim\sum_{j=0}^\infty h^jq_j(x,\xi)\quad\text{as }h\to 0.
$$
The $L^2$ boundedness of $Q_I(s)$ follows from Lemma~\ref{l:B-s-bdd},
whose proof applies to $\mathcal B(1-s)$.

By Lemma~\ref{l:bi-helper}, we have
$
\psi_I\mathcal B(s)\omega_IQ_I(s)=\Op_h(b)
$
where the symbol $b$ satisfies
$$
b(x,\xi;h)\sim\sum_{j=0}^\infty h^jb_j(x,\xi),\quad
b_j(x,\xi)={1\over 2}q_j(x,\xi)+\dots
$$
and `$\dots$' denotes terms depending on $q_0,\dots,q_{j-1}$
according to~\eqref{e:bi-helper-exp}, and in particular
supported in $\supp a$.
Here we use that on $\supp q_j\subset \supp a$
we have $\psi_I(x)=1$ and
$\omega_I(x'(x,\xi))=1$ by~\eqref{e:bi-1}.
We can now iteratively construct
$q_0,q_1,\dots$ such that $b=a+\mathcal O(h^\infty)_{\mathscr S(\mathbb R^2)}$, finishing the proof.
\end{proof}


\end{document}